\documentclass[10pt,a4paper]{article}
\usepackage[T1]{fontenc}
\usepackage[utf8]{inputenc}

\usepackage[margin=2cm]{geometry}

\usepackage{amsmath}
\usepackage{amssymb}
\usepackage{amsthm}
\usepackage{amsfonts}
\usepackage{stmaryrd}

\usepackage{graphicx}

\usepackage{bm}
\usepackage[svgnames]{xcolor}

\usepackage[
backend=biber,
style=numeric-comp,
giveninits=true,
]{biblatex}
\DeclareFieldFormat*{title}{#1}
\renewbibmacro{in:}{
    \ifentrytype{article}{}{\printtext{\bibstring{in}\intitlepunct}}}

\usepackage[colorlinks]{hyperref}

\newtheorem{theorem}{Theorem}[section]
\newtheorem{lemma}[theorem]{Lemma}
\newtheorem{corollary}[theorem]{Corollary}
\newtheorem{problem}[theorem]{Problem}

\newcommand{\scalar}[2]{\left(#1, #2\right)}
\newcommand{\norm}[1]{\lVert #1 \rVert}
\newcommand{\average}[1]{\lbrace \kern-0.75ex \lbrace #1 \rbrace \kern-0.75ex \rbrace}

\DeclareMathOperator{\divv}{div}
\DeclareMathOperator{\tang}{tang}
\DeclareMathOperator{\diag}{diag}

\begin{document}
    \title{Analysis of a family of time-continuous strongly conservative space-time finite element methods for the dynamic Biot model}
    \author{Johannes Kraus\thanks{Faculty of Mathematics, University of Duisburg-Essen, Germany}
        \and
        Maria Lymbery\thanks{Institute of Artificial Intelligence in Medicine, University of Duisburg-Essen, Germany} \and
        Kevin Osthues\footnotemark[1] \and
        Fadi Philo\footnotemark[1]}
    
    \maketitle
    
    \begin{abstract}
        We consider the dynamic Biot model~\cite{Biot1962Mech} describing the interaction between fluid flow and solid deformation
        including wave propagation phenomena in both the liquid and solid phases of a saturated porous medium. The model couples
        a hyperbolic equation for momentum balance to a second-order in time dynamic Darcy law and a parabolic equation for the
        balance of mass and is here considered in three-field formulation with the displacement of the elastic matrix, the fluid velocity,
        and the fluid pressure being the physical fields of interest.
        
        A family of variational space-time finite element methods is proposed that combines a continuous-in-time Galerkin ansatz of arbitrary
        polynomial degree with inf-sup stable $H(\divv)$-conforming approximations of discontinuous Galerkin (DG) type in case of the
        displacement and a mixed approximation of the flux, its time derivative and the pressure field. We prove error estimates in a
        combined energy norm as well as $L^2$~error estimates in space for the individual fields for both maximum and $L^2$ norm in time
        which are optimal for the displacement and pressure approximations.
    \end{abstract}
    
    \begin{quote}
    \end{quote}

    \section{Introduction}
    The theory of poroelasticity dates back to mid of the last century and was most decisively developed by Maurice Anthony Biot in a series of
    pioneering papers, see~\cite{Biot1941general,Biot1955theory,Biot1956theory,Biot1956waves1,Biot1956waves2}. Starting with a quasistatic
    model of consolidation~\cite{Biot1941general}, the theory was generalized from isotropic to anisotropic media~\cite{Biot1955theory}, viscoelastic
    materials~\cite{Biot1956theory}, and extended to dynamics~\cite{Biot1956waves1,Biot1956waves2}.   
    The dynamic models in~\cite{Biot1956waves1,Biot1956waves2,Biot1962acoust,Biot1962Mech} are derived from general principles of nonequilibrium
    thermodynamics and describe the interaction of fluid flow and deformation of (linear, anisotropic) porous media taking into account also wave
    propagation, energy dissipation, and other relaxation effects.
    
    First fundamental mathematical results related to the well-posedness, discretization and stability of the models proposed by Biot
    were addressed, e.g., in~\cite{Carlson1973Lin,Dafermos1968exist,Showalter2000Diff,Zienkiewicz1982Basic,Zienkiewicz1984dynamic,
        Zenisek1984existence,Mielke2013Homo}, for a more comprehensive overview see also~\cite{Cheng2016Poro,Seifert2022Evol}.    
    
    The classical area of application of (thermo-)poroelasticity is geomechanics where typical problems reach from evaluating induced subsidence
    related to hydrocarbon production, simulation of underground reservoirs, environmental planning to the engineering of safe barriers for radioactive
    waste disposals, see,  e.g.,~\cite{Barenblatt1960basic,Zenisek1984finite,Bjornara2016vertically,Michalec2021fully,Gao2022coupled}. 
    In recent years, the important observation that biological tissues can be interpreted as porous, permeable and deformable media infiltrated by fluids, 
    such as blood and interstitial fluid, cf. e.g.~\cite{BOTH2022114183, Pitre2017}, has ultimately led to the successful employment of poroelasticity models
    in the studies of the biomechanical behavior of different human organs, such as the brain,
    \cite{Guo_etal2018subject-specific,Vardakis2019fluid,Kraus2023hybridized}, 
    the liver, \cite{Aichele2021fluids}, 
    the kidney, \cite{Ong2009Deformation},
    which have offered better understanding of various disease states ranging from renal failure to traumatic brain injury to cancer.

    The range of physical parameters in the models under consideration, especially in biomedical applications, makes the design and
    analysis of adequate numerical methods a challenging task. This is why significant effort has been put into developing stable and
    parameter-robust methods for the numerical solution of the underlying coupled partial differential algebraic equations over the last decades. 
    
    A rigorous stability and convergence analysis of finite element approximations of the two-field formulation of Biot's model of consolidation
    in terms of the displacement field and the fluid pressure has first been presented in~\cite{MuradLoula1992improved,MuradLoula1994stability}.
    The derived a priori error estimates are valid for both semidiscrete and fully discrete formulations, where inf-sup stable finite elements are used
    for space discretization and the backward Euler method is employed for time-stepping. Moreover, a post‐processing technique has been introduced to 
    improve the pore-pressure accuracy.
    Numerous works have addressed various aspects of locking-free and (locally) conservative discretization of this class of problems since then,
    more lately, also using generalized and extended finite element, see, e.g.,~\cite{Fries2010extended} and enriched Galerkin methods, see,
    e.g.,~\cite{Lee2018enriched,Lee2023locking}, and to study flow and wave propagation phenomena in fractured heterogeneous porous
    media, see~\cite{Janaki2018enriched}.
    Optimal $L^2$ error estimates for the quasistatic Biot problem have been derived only recently, see~\cite{Wheeler2022optimal}.
    
    Discretizations of the quasistatic Biot problem in classic three-field-formulation, using standard continuous Galerkin approximations
    of the displacement field and a mixed approximation of the flux-pressure pair, where the flux variable is introduced to the system via
    the standard Darcy law, have originally been proposed in~\cite{PhillipsWheeler2007coupling1,PhillipsWheeler2007coupling2}.
    The error estimates proved in these works are for both continuous-  and discrete-in-time approximations.
    This approach has also been extended to discontinuous Galerkin approximations of the displacement field,
    see~\cite{PhillipsWheeler2008coupling}, and other nonconforming approximations, e.g., using modified rotated bilinear
    elements~\cite{Yi2013coupling}, or Crouzeix-Raviart elements for the displacements, see~\cite{Hu2017nonconforming}.
    More recently, in~\cite{Hong2018parameter}, a family of strongly mass conserving discretizations based on $H(\text{div})$-conforming
    discontinuous Galerkin discretization of the displacement field has been proven to be uniformly well-posed in properly 
    fitted parameter-dependent norms resulting in near best
    approximation results robust with respect to all model and discretization parameters (Lam\'{e} parameters, permeability and Biot-Willis
    parameter, storage coefficient, time step and mesh size). Error estimates for the related continuous- and discrete-in-time scheme
    have been delivered in~\cite{KanschatRiviere2018finite} and hybridization techniques considered in~\cite{Kraus2021Unif}.
    A stabilized conforming method for the Biot problem in three-field formulation was proposed in~\cite{Rodrigo2018new}.
    
    The four-field formulation of the quasi-static Biot problem in~\cite{Yi2014convergence} couples two mixed problems and uses a
    mixed finite element method based on the Hellinger-Reissner principle for the mechanics subproblem and a standard mixed finite
    element method for the flow subproblem.
    Optimal a priori error estimates have been proved for both semidiscrete and fully discrete schemes using the Arnold-Winther
    space and the Raviart-Thomas space in the two coupled mixed finite element methods, respectively. The same discretization for
    the same formulation has  further been analyzed in~\cite{Lee2016robust} providing error estimates in $L^\infty$ norm in time and
    $L^2$ norm in space, which are robust with respect to the Lam\'{e} parameters and remain valid in the limiting case of vanishing
    constrained storage coefficient.
    
    Several of these techniques have also been transferred to multiple network poroelasticity theory (MPET), with a focus on
    the stability of continuous and discrete models, mass conservation, hybridization, near-best approximation results, uniform
    preconditioning, and splitting schemes, see, e.g.~\cite{HongEtAl2020parameter,HoKrLyPh2020,Kraus2021Unif}.
    Other three-, four-, and five-field models additionally involve a total pressure and their discretization and a priori error analysis
    as well as robust preconditioning have been addressed,
    e.g., in~\cite{Oyarzua2016locking,Lee2017parameter,Kumar2020conservative,Lee2023analysis}.
    
    Most of the works mentioned so far--except some references in the first two paragraphs--propose and analyze numerical methods
    for quasi-static models only. While these have been intensively studied over the years and are quite well understood, far less is
    known about stable numerical solution methods for the dynamic Biot model. However, there are a few early groundbraking theoretical
    results on the design and error analysis of finite element approximations for the latter class of problems, see,
    e,g.,~\cite{Dupont1973L2,Baker1976error,Wheeler1973Linfty}. 
    
    The contribution of the present paper is a rigorous convergence analysis of a family of variational space-time finite element methods
    for the dynamic Biot model as presented in~\cite{Biot1956waves1,Biot1956waves2,Biot1962Mech}.
    The continuous-in time-Petrov Galerkin scheme that we use for time discretization has also been studied in~\cite{Bause2023Conv}
    but for a different hyperbolic-parabolic problem in two-field formulation. In fact, the dynamic model in three-field formulation, which
    is subject to the error analysis presented in this work, describes the dynamics of the solid and fluid phases.
    The model in~\cite{Bause2023Conv} follows from system~\eqref{eq:dg_dyn_biot} by setting the fluid density and herewith the
    effective fluid density to zero resulting in a simplification of the momentum balance equation~\eqref{eq:dg_dyn_biot-1}. Moreover,
    the mass balance~\eqref{eq:dg_dyn_biot-3} uses a flux introduced via the dynamic Darcy law~\eqref{eq:dg_dyn_biot-2}, contrary
    to the model in~\cite{Bause2023Conv}, which one would obtain by substituting the standard Darcy 
    flux in the mass balance.
    
    Another basic difference between the approaches studied in~\cite{Bause2023Conv} and here is, that we employ
    $H(\divv)$-conforming approximations, conforming for fluxes and nonconforming of discontinuous Galerkin type
    for the displacement field, thus following~\cite{Hong2019conservativeMPET}.
    Moreover, contrary to~\cite{Bause2023Conv}, we consider a classic three-field formulation, driven by
    the interest to model both the dynamics of the solid and fluid phases.
    The finite element spaces for space-discretizations are chosen in such a way that they ensure point-wise conservation of mass,
    as in~\cite{Hong2018parameter,Hong2019conservativeMPET}. The proofs we present here are inspired by and use ideas
    from~\cite{Dupont1973L2,Baker1976error,Wheeler1973Linfty}. 
    On the bottom line, the difference in the
    dynamic model, as compared to~\cite{Bause2023Conv}, using a three-field instead of a two-field formulation, and the combination of continuous and discontinuous Galerkin
    approximations in space require a different splitting of the error and change the intermediate estimates in the error analysis
    considerably.  
    An analysis, which uses similar arguments as the ones presented here, but for the Crank-Nicolson scheme in time,
    which is equivalent to the lowest-order time discretization considered here, has been performed in~\cite{Philo2022phd}.
    
    The remainder of the manuscript is organized as follows: In Section~\ref{sec:2} we state the three-field formulation of the dynamic Biot 
    problem under consideration. Section~\ref{sec:3} summarizes some of the notation used in this paper, recalls quadrature formulas, 
    introduces the interpolation operators and some related useful auxiliary results. Section~\ref{sec:4} presents the variational 
    formulations giving the starting point for the discrete models and the forthcoming analysis. Section~\ref{sec:5} contains 
    a detailed derivation of the error estimates which are finally given in Theorem~\ref{thm:dg_main_result}, which is the main
    theoretical result of this work.

    \section{Problem formulation}\label{sec:2}
    In this paper, we study the dynamic Biot model for a linear elastic, homogeneous, isotropic, porous solid saturated with a slightly compressible fluid as proposed in~\cite{Biot1962Mech, Biot1941general}, that is, the following coupled system of partial differential equations
    \begin{subequations}
        \label{eq:dg_dyn_biot}
        \begin{alignat}{2}
            \bar{\rho}  \partial_{t}^{2} \bm{u} - 2 \mu \divv(\bm{\epsilon}(\bm{u})) - \lambda  \nabla \divv(\bm{u}) + \alpha  \nabla p + \rho_{f}  \partial_{t} \bm{w} & = \bm{f} & \quad \text{in $ \Omega \times \left(0, T\right] $}, \label{eq:dg_dyn_biot-1} \\
            \rho_{f}  \partial_{t}^{2} \bm{u} + \rho_{w}  \partial_{t} \bm{w} + \bm{K}^{-1} \bm{w} + \nabla p & = \bm{g} & \quad \text{in $ \Omega \times \left(0, T\right] $}, \label{eq:dg_dyn_biot-2} \\
            s_{0}  \partial_{t} p + \alpha \divv(\partial_{t} \bm{u}) + \divv(\bm{w}) & = 0 & \quad \text{in $ \Omega \times \left(0, T\right] $}, \label{eq:dg_dyn_biot-3}
        \end{alignat}
    \end{subequations}
    in a bounded Lipschitz domain $ \Omega \subset \mathbb{R}^{d},  d \in \{2, 3\} $ and $T>0$.
    The physical parameters in the model are
    the total density $ \bar{\rho} = (1 - \phi_{0})  \rho_{s} + \phi_{0}  \rho_{f} > 0 $ and
    the effective fluid density $\rho_{w} \geq \phi_{0}^{-1} \rho_{f} > 0 $, where $\rho_{s} > 0 $ represents the solid density, $ \rho_{f} > 0$ the fluid density and $\phi_{0} \in (0, 1) $ the porosity.
    In addition, the Biot-Willis parameter is specified by $ \alpha \in \left[\phi_{0}, 1\right]$, cf.~\cite{Biot1957Thee} and the constrained specific storage coefficient by $ s_{0} > 0 $, cf.~\cite{Biot1962Mech}.
    The material-dependent Lam\'{e} parameters are denoted by $ \lambda > 0 $ and $ \mu > 0 $.
    Furthermore, $ \bm{K} $ represents the symmetric positive definite permeability tensor, $ \bm{f}: \Omega \times [0, T] \to \mathbb{R}^{d} $ the body force density and $  \bm{g}: \Omega \times [0, T] \to \mathbb{R}^{d} $ a source term.
    The strain tensor is given by the symmetric part of the gradient of $\bm{u}$, that is $\bm{\epsilon}(\bm{u}) = (\nabla \bm{u} + (\nabla \bm{u})^{\top}) / 2  $.
    
    The dynamic Biot model comprises the momentum balance equation \eqref{eq:dg_dyn_biot-1}, the dynamic Darcy law \eqref{eq:dg_dyn_biot-2} and the mass balance equation \eqref{eq:dg_dyn_biot-3}, see \cite{Mielke2013Homo}.
    Therein, the unknown physical fields are described by the displacement $ \bm{u} = \bm{u}(\bm{x}, t) $, fluid velocity $ \bm{w} = \bm{w}(\bm{x}, t) $ and fluid pressure $ p = p(\bm{x}, t) $.
    By neglecting acceleration of the solid and the fluid phase, the dynamic Biot model~\eqref{eq:dg_dyn_biot} reduces to the classical three-field formulation of the quasi-static Biot problem~\cite{Hong2018parameter}.
    We complete problem~\eqref{eq:dg_dyn_biot} by the following initial and boundary conditions
    \begin{alignat}{4}
        \label{eq:dg_boundary_conditions}
        \begin{aligned}
            \bm{u} & = \bm{u}_{0} & \quad & \text{in $ \Omega \times \{0\} $}, \qquad\qquad & 
            \bm{u} & = \bm{0} & \quad & \text{on $ \partial \Omega \times \left(0, T\right] $}, \\
            \partial_{t} \bm{u} & = \bm{v}_{0} & \quad & \text{in $ \Omega \times \{0\} $}, &
            \bm{w} \cdot \bm{n} & = 0 & \quad & \text{on $ \partial \Omega \times \left(0, T\right] $}, \\
            p & = p_{0} & \quad & \text{in $ \Omega \times \{0\} $}.
        \end{aligned}
    \end{alignat}
    Subject to the initial and boundary conditions \eqref{eq:dg_boundary_conditions}, problem~\eqref{eq:dg_dyn_biot}  is well-posed which can be shown using Picard's theorem~\cite[Thm.~6.2.1]{Seifert2022Evol}.

    \section{Preliminaries}\label{sec:3}
    
    \subsection{Function spaces}
    In this paper, we use standard notation for Sobolev spaces.
    The $L^{2}$ inner product is denoted by $\scalar{\cdot}{\cdot}$.
    For our analysis, we need the Bochner spaces $ L^{2}(J; B) $, $ C(J; B) $, $ C^{m}(J; B),  m \in \mathbb{N} $, where $ B $ is a Banach space and $ J \subseteq [0, T] $.
    Throughout the remainder, we mark vector- and matrix-valued quantities in bold.
    
    In view of the proposed time stepping method that is continuous in time, we split the time interval $ I := \left(0, T\right] $ into equidistant subintervals $ I_{n} := \left( t_{n-1}, t_{n} \right],  n = 1, \ldots, N $ 
    of length $\tau$,
    where $ 0 = t_{0} < t_{1} < \cdots < t_{N} = T $.
    Moreover, we define the time mesh by $ \mathcal{M}_{\tau} := \{I_{1}, \ldots, I_{N}\} $. For a fixed $ k \in \mathbb{N}_{0} $ and Banach space $ B $, we introduce the space of $ B $-valued polynomials in time by
    \begin{align*}
        \mathbb{P}_{k}(I_{n}; B) := \left\{ f_{\tau} : I_{n} \to B : f_{\tau}(t) = \sum_{i=0}^{k} b_{i}  t^{i}, \; b_{i} \in B \ \forall t \in I_{n} \; \forall i \right\}.
    \end{align*}
    Additionally, we define the space of globally continuous function in time by
    \begin{align*}
        \mathbb{X}_{\tau}^{k}(B) & := \left\{f_{\tau} \in C(\overline{I}; B) : f_{\tau} \big|_{I_{n}} \in \mathbb{P}_{k}(I_{n}; B) \quad \forall I_{n} \in \mathcal{M}_{\tau}\right\},
    \end{align*}
    and the space of $L^{2}$ functions in time by
    \begin{align*}
        \mathbb{Y}_{\tau}^{k}(B) & := \left\{f_{\tau} \in L^{2}(I; B) : f_{\tau} \big|_{I_{n}} \in \mathbb{P}_{k}(I_{n}; B) \quad \forall I_{n} \in \mathcal{M}_{\tau}\right\}.
    \end{align*}

    \subsection{Quadrature formulas and interpolation operators}
    \label{subsec:quadrature_formulas}
    Here, we review some quadrature formulas and interpolation operators needed throughout the remainder. The $ (k+1) $-point Gauss-Lobatto 
    quadrature formula on $ I_{n} = \left(t_{n-1}, t_{n}\right] $ is given by
    \begin{align}
        \label{eq:quadrature_gl}
        Q_{n}^{\text{GL}}(f) := \tau \sum_{i=0}^{k} \hat{\omega}_{i}^{\text{GL}}  f\big|_{I_{n}}(t_{n, i}^{\text{GL}})
        \approx \int_{I_{n}} f(t) \mathrm{d}t,
    \end{align}
    where $ t_{n, i}^{\text{GL}} := T_{n}(\hat{t}_{i}^{\text{GL}}),  i = 0, \ldots, k $ are the quadrature points on interval $ \overline{I}_{n} $ with quadrature weights $ \hat{\omega}_{i}^{\text{GL}} $.
    The quadrature points $ \{t_{n, i}^{\text{GL}}\}_{i=0}^{k} $ are obtained via
    \begin{align}
        T_{n} & : \hat{I} \to I_{n}, &
        T_{n}(\hat{t}) & = \tau  \hat{t} + t_{n-1},
    \end{align}
    applied to the Gauss-Lobatto quadrature points $ \hat{t}_{i}^{\text{GL}} $ on reference interval $ \hat{I} := [0, 1] $.
    
    Furthermore, we need the $ k $-point Gauss quadrature formula on $ I_{n} $ defined by
    \begin{align}
        \label{eq:quadrature_g}
        Q_{n}^{\text{G}}(f) := \tau \sum_{i=1}^{k} \hat{\omega}_{i}^{\text{G}}  f\big|_{I_{n}}(t_{n, i}^{\text{G}})
        \approx \int_{I_{n}} f(t) \mathrm{d}t,
    \end{align}
    where $ t_{n, i}^{\text{G}} := T_{n}(\hat{t}_{i}^{\text{G}}) $ and $ \hat{\omega}_{i}^{\text{G}},  i = 1, \ldots, k $ denote the quadrature points and weights, 
    respectively, with $ \{\hat{t}_{i}^{\text{G}}\}_{i=1}^{k} $ being the quadrature points on reference interval $ \hat{I} $.
    Both quadrature formulas are exact for polynomials of degree $ 2k-1 $.
    
    In our analysis, we make use of different Lagrange interpolations operators.
    To this end, we first define local operators on subinterval $ I_{n} $ via
    \begin{alignat*}{3}
        \mathcal{I}_{\tau, n}^{\text{GL}} & : C(\overline{I}_{n}; B) \to \mathbb{P}_{k}(\overline{I}_{n}; B), & \qquad
        \mathcal{I}_{\tau, n}^{\text{GL}} f(t_{n, i}^{\text{GL}}) & = f(t_{n, i}^{\text{GL}}), & \quad i & = 0, \ldots, k, \\
        \mathcal{I}_{\tau, n}^{\text{G}} & : C(\overline{I}_{n}; B) \to \mathbb{P}_{k-1}(\overline{I}_{n}; B), & \qquad
        \mathcal{I}_{\tau, n}^{\text{G}} f(t_{n, i}^{\text{G}}) & = f(t_{n, i}^{\text{G}}), & \quad i & = 1, \ldots, k, \\
        \mathcal{I}_{\tau, n}^{\text{G}, 0} & : C(\overline{I}_{n}; B) \to \mathbb{P}_{k}(\overline{I}_{n}; B), & \qquad
        \mathcal{I}_{\tau, n}^{\text{G}, 0} f(t_{n, i}^{\text{G}, 0}) & = f(t_{n, i}^{\text{G}, 0}), & \quad i & = 0, \ldots, k,
    \end{alignat*}
    where $ B $ is a Banach space, $ t_{n, 0}^{\text{G}, 0} := t_{n-1} $ and $ t_{n, i}^{\text{G}, 0} := t_{n, i}^{\text{G}} $ for $ i = 1, \ldots, k $.
    With these local interpolation operators at hand, we can specify the global interpolation operators by
    \begin{alignat*}{3}
        \mathcal{I}_{\tau}^{\text{GL}} & : C(\overline{I}; B) \to \mathbb{X}_{\tau}^{k}(B), & \qquad
        \mathcal{I}_{\tau}^{\text{GL}} f\big|_{I_{n}} & = \mathcal{I}_{\tau, n}^{\text{GL}}(f\big|_{I_{n}}), & \quad n & = 1, \ldots, N, \\
        \mathcal{I}_{\tau}^{\text{G}} & : C(\overline{I}; B) \to \mathbb{Y}_{\tau}^{k-1}(B), & \qquad
        \mathcal{I}_{\tau}^{\text{G}} f\big|_{I_{n}} & = \mathcal{I}_{\tau, n}^{\text{G}}(f\big|_{I_{n}}), & \quad n & = 1, \ldots, N, \\
        \mathcal{I}_{\tau}^{\text{G}, 0} & : C(\overline{I}; B) \to \mathbb{Y}_{\tau}^{k}(B), & \qquad
        \mathcal{I}_{\tau}^{\text{G}, 0} f\big|_{I_{n}} & = \mathcal{I}_{\tau, n}^{\text{G}, 0}(f\big|_{I_{n}}), & \quad n & = 1, \ldots, N.
    \end{alignat*}
    Note that the identity
    \begin{align}
        \label{eq:representation_I_tau}
        \mathcal{I}_{\tau}^{\text{G}, 0}\big|_{I_{n}} y(t) = \sum_{i=0}^{k} y(t_{n, i}^{\text{G}, 0})  L_{n, i}^{\text{G}, 0}(t) \qquad \text{with} \qquad
        L_{n, i}^{\text{G}, 0}(t) := \prod_{\substack{j=0 \\ j \neq i}}^{k} \frac{t - t_{n, j}^{\text{G}, 0}}{t_{n, i}^{\text{G}, 0} - t_{n, j}^{\text{G}, 0}}, \quad \text{for $ i = 0, \ldots, k $},
    \end{align}
    is valid for $ y \in C(\overline{I}_{n}; B) $.
    Analogously, we define the Lagrange interpolation polynomials $ \{L_{n, i}^{\text{G}}\}_{i=1}^{k} $ corresponding to the quadrature points $ \{t_{n, i}^{\text{G}}\}_{i=1}^{k} $.
    We will further use the notation $ \mathcal{I}_{\tau}^{*} $ for any Lagrange interpolation operator
    defined for $ k+1 $ points $ t_{n, i}^{*},  i = 0, \ldots, k $.
    Recall that on each subinterval $ I_{n} $, the Lagrange interpolation operator $ \mathcal{I}_{\tau}^{*} $ fulfills the estimates
    \begin{subequations}
        \label{eq:interpolation_error_GL}
        \begin{align}
            \norm{\mathcal{I}_{\tau}^{*} f - f}_{L^{2}(I_{n}; L^{2})} & \leq c \tau^{k+1}  \norm{\partial_{t}^{k+1} f}_{L^{2}(I_{n}; L^{2})}, \label{eq:interpolation_error_GL-1} \\
            \norm{\mathcal{I}_{\tau}^{*} f}_{L^{2}(I_{n}; L^{2})} & \leq c \tau^{k+1}  \norm{\partial_{t}^{k+1} f}_{L^{2}(I_{n}; L^{2})} + \norm{f}_{L^{2}(I_{n}; L^{2})}. \label{eq:interpolation_error_GL-2}
        \end{align}
    \end{subequations}

    \subsection{Auxiliary results}
    \label{subsec:auxiliary_results}
    In this subsection, we present some results that are useful for the error analysis performed in Section~\ref{sec:5}.
    Their proofs can be found in appendices~\ref{sec:proof_auxiliary_I}--\ref{sec:proof_auxiliary_III}.
    \begin{lemma}
        \label{lem:dg_auxiliary_I}
        Let
        \begin{align}
            \label{eq:dg_auxiliary_I_assumption}
            x & := \sum_{i=0}^{k} x_{i}  L_{n, i}^{\textnormal{G}, 0}, &
            \bar{x} & := \sum_{i=1}^{k} x_{i}  L_{n, i}^{\textnormal{G}}, &
            y & := \sum_{i=0}^{k} y_{i}  L_{n, i}^{\textnormal{G}, 0}.
        \end{align}
        Then the following statements
        \begin{subequations}
            \label{eq:dg_auxiliary_I}
            \begin{align}
                \int_{I_{n}} (\sum_{i=0}^{k} \beta_{i} x_{i}  L_{n, i}^{\textnormal{G}, 0}, z) \mathrm{d}t
                & = \int_{I_{n}} (\sum_{i=1}^{k} \beta_{i} x_{i}  L_{n, i}^{\textnormal{G}}, z) \mathrm{d}t \qquad \forall z \in \mathbb{P}_{k-1}(I_{n}; L^{2}), \label{eq:dg_auxiliary_I-1} \\
                \int_{I_{n}} (x, \sum_{i=1}^{k} \beta_{i} y_{i}  L_{n, i}^{\textnormal{G}}) \mathrm{d}t
                & = \int_{I_{n}} (\sum_{j=1}^{k} \beta_{j} x_{j}  L_{n, j}^{\textnormal{G}}, y) \mathrm{d}t, \label{eq:dg_auxiliary_I-2} \\
                \norm{\sum_{i=1}^{k} \beta_{i} x_{i}  L_{n, i}^{\textnormal{G}}}_{L^{2}(I_{n}; L^{2})}
                & \leq c \norm{x}_{L^{2}(I_{n}; L^{2})}, \label{eq:dg_auxiliary_I-3} \\
                \norm{\bar{x}}_{L^{2}(I_{n}; L^{2})} & \leq c \norm{x}_{L^{2}(I_{n}; L^{2})}, \label{eq:dg_auxiliary_I-5}
            \end{align}
        \end{subequations}
        are satisfied, where $ \beta_{i} \in \mathbb{R} $ for $ i = 0, \ldots, k $.
    \end{lemma}
    
    \begin{proof}
        The proofs are presented in Appendix~\ref{sec:proof_auxiliary_I}.
    \end{proof}
    
    Subsequently, we list two identities for a specific choice of polynomials.
    
    \begin{lemma}
        \label{lem:dg_auxiliary_II}
        Let
        \begin{align}
            \label{eq:dg_auxiliary_II_assumption}
            \begin{aligned}
                x(t) & := \sum_{i=0}^{k} x_{i}  L_{n, i}^{\textnormal{G}, 0}(t), & \qquad
                x_{\beta}(t) & := \sum_{i=0}^{k} \sum_{j=1}^{k} \beta_{j} x_{i}  \partial_{t} L_{n, i}^{\textnormal{G}, 0}(t_{n, j}^{\textnormal{G}})  L_{n, j}^{\textnormal{G}}(t), \\
                y(t) & := \sum_{i=0}^{k} y_{i}  L_{n, i}^{\textnormal{G}, 0}(t), & \qquad
                y_{\beta}(t) & := \sum_{i=0}^{k} \sum_{j=1}^{k} \beta_{j} y_{i}  \partial_{t} L_{n, i}^{\textnormal{G}, 0}(t_{n, j}^{\textnormal{G}})  L_{n, j}^{\textnormal{G}}(t).
            \end{aligned}
        \end{align}
        and $ \beta_{j} \in \mathbb{R} $ for $ j = 0, \ldots, k $. Then the identities
        \begin{subequations}
            \label{eq:dg_auxiliary_II}
            \begin{align}
                \int_{I_{n}} (x, y_{\beta}) \mathrm{d}t & = \int_{I_{n}} (\sum_{j=0}^{k} \beta_{j} x_{j}  L_{n, j}^{\textnormal{G}, 0}, \partial_{t} y) \mathrm{d}t, \label{eq:dg_auxiliary_II-1} \\
                \int_{I_{n}} (\partial_{t} x, y_{\beta}) \mathrm{d}t & = \int_{I_{n}} (x_{\beta}, \partial_{t} y) \mathrm{d}t, \label{eq:dg_auxiliary_II-2}
            \end{align}
        \end{subequations}
        hold true.
    \end{lemma}
    
    \begin{proof}
        For a proof of these equalities, see Appendix~\ref{sec:proof_auxiliary_II}.
    \end{proof}
    
    Finally, we summarize some important estimates for our error analysis in Section~\ref{sec:5}.
    
    \begin{lemma}
        \label{lem:dg_auxiliary_III}
        Let
        \begin{align}
            \label{eq:dg_auxiliary_III_assumption}
            \begin{aligned}
                x(t) & = \sum_{i=0}^{k} x_{i}  L_{n, i}^{\textnormal{G}, 0}(t), & \qquad
                x_{\beta}(t) & = \sum_{i=0}^{k} \sum_{j=1}^{k} \beta_{j} x_{i}  \partial_{t} L_{n, i}^{\textnormal{G}, 0}(t_{n, j}^{\textnormal{G}})  L_{n, j}^{\textnormal{G}}(t).
            \end{aligned}
        \end{align}
        Furthermore, we set $ \beta_{i} := (\hat{t}_{i}^{\textnormal{G}})^{-1} $ for $ i = 1, \ldots, k $ and $ \beta_{0} := 1 $.
        Then the following statements
        \begin{subequations}
            \label{eq:dg_auxiliary_III}
            \begin{align}
                C_{1} \norm{x}_{L^{2}(I_{n}; L^{2})}^{2}
                & \leq \tau \int_{I_{n}} \scalar{\sum_{i=0}^{k} \beta_{i} x_{i}  L_{n, i}^{\textnormal{G}, 0}}{\partial_{t} x} \mathrm{d}t + C_{2} \tau \norm{x(t_{n-1}^{+})}_{L^{2}(\Omega)}^{2}, \label{eq:dg_auxiliary_III-1} \\
                C_{1} \norm{x}_{L^{2}(I_{n}; L^{2})}^{2} & \leq \int_{I_{n}} \scalar{\sum_{i=1}^{k} \beta_{i} x_{i}  L_{n, i}^{\textnormal{G}}}{x} \mathrm{d}t + C_{2} \tau \norm{x(t_{n-1}^{+})}_{L^{2}(\Omega)}^{2}, \label{eq:dg_auxiliary_III-2} \\
                \int_{I_{n}} \scalar{\partial_{t} x}{x_{\beta}} \mathrm{d}t & \simeq \norm{\partial_{t} x}_{L^{2}(I_{n}; L^{2})}^{2}, \label{eq:dg_auxiliary_III-3} \\
                \norm{x_{\beta}}_{L^{2}(I_{n}; L^{2})} & \simeq \norm{\partial_{t} x}_{L^{2}(I_{n}; L^{2})}, \label{eq:dg_auxiliary_III-4} \\
                \norm{\partial_{t} x}_{L^{2}(I_{n}; L^{2})} & \leq C_{3} \tau^{-1} \norm{x}_{L^{2}(I_{n}; L^{2})}, \label{eq:dg_auxiliary_III-5}
            \end{align}
        \end{subequations}
        are valid.
    \end{lemma}
    
    \begin{proof}
        The results are proven in Appendix~\ref{sec:proof_auxiliary_III}.
    \end{proof}

    \section{Variational formulations}\label{sec:4}
    To begin with, we introduce the variable $ \bm{v} := \partial_{t} \bm{u} $ and rewrite problem \eqref{eq:dg_dyn_biot} as
    \begin{subequations}
        \label{eq:dg_dyn_biot_velo}
        \begin{alignat}{2}
            \bar{\rho}  \partial_{t} \bm{v} - 2 \mu \divv(\bm{\epsilon}(\bm{u})) - \lambda  \nabla \divv(\bm{u}) + \alpha  \nabla p + \rho_{f}  \partial_{t} \bm{w} & = \bm{f} & \quad \text{in $ \Omega \times \left(0, T\right] $}, \label{eq:dg_dyn_biot_velo-1} \\
            \partial_{t} \bm{u} - \bm{v} & = \bm{0} & \quad \text{in $ \Omega \times \left(0, T\right] $}, \label{eq:dg_dyn_biot_velo-2} \\
            \rho_{f}  \partial_{t} \bm{v} + \rho_{w}  \partial_{t} \bm{w} + \bm{K}^{-1} \bm{w} + \nabla p & = \bm{g} & \quad \text{in $ \Omega \times \left(0, T\right] $}, \label{eq:dg_dyn_biot_velo-3} \\
            s_{0}  \partial_{t} p + \alpha \divv(\partial_{t} \bm{u}) + \divv(\bm{w}) & = 0 & \quad \text{in $ \Omega \times \left(0, T\right] $}. \label{eq:dg_dyn_biot_velo-4}
        \end{alignat}
    \end{subequations}
    In the following, we present a time-continuous as well as a time-discrete variational formulation of problem~\eqref{eq:dg_dyn_biot_velo}.
    For the discrete problem, we utilize a discontinuous Galerkin methods in space.

    \subsection{Space-time variational problem}
    In order to give a space-time variational formulation of problem~\eqref{eq:dg_dyn_biot_velo}, we define the following spaces
    \begin{align}
        \begin{aligned}
            \bm{U} = \bm{V} & := \bm{H}_{0}^{1}(\Omega) = \{ \bm{u} \in \bm{H}^{1}(\Omega) : \bm{u} = \bm{0} \text{ on } \partial \Omega \}, \\
            \bm{W} & := \bm{H}_{0}(\divv, \Omega) = \{ \bm{w} \in \bm{H}(\divv, \Omega) : \bm{w} \cdot \bm{n} = 0 \text{ on } \partial \Omega \}, \\
            P & := L_{0}^{2}(\Omega)
            = \{ p \in L^{2}(\Omega) : \int_{\Omega} p  \mathrm{d}\bm{x} = 0 \}, \\
            \bm{X} & := \bm{U} \times \bm{V} \times \bm{W} \times P,
        \end{aligned}
    \end{align}
    and bilinear forms
    \begin{subequations}
        \begin{align}
            a(\bm{u}, \bm{\varphi}^{\bm{u}}) & := 2 \mu \scalar{\bm{\epsilon}(\bm{u})}{\bm{\epsilon}(\bm{\varphi}^{\bm{u}})} + \lambda \scalar{\divv(\bm{u})}{\divv(\bm{\varphi}^{\bm{u}})}, \\
            b(p, \bm{\varphi}^{\bm{w}}) & := \scalar{p}{\divv(\bm{\varphi}^{\bm{w}})},
        \end{align}
    \end{subequations}
    for $ \bm{u}, \bm{\varphi}^{\bm{u}} \in \bm{U}, \bm{\varphi}^{\bm{w}} \in \bm{W}, p \in P$. Next, we pose the time-continuous weak problem:
    
    \begin{problem}
        \label{pr:dg_continuous_weak_formulation}
        Let $ \bm{f}, \bm{g} \in W^{1, 1}([0, T]; \bm{L}^{2}(\Omega)) $ and $ (\bm{u}_{0}, \bm{v}_{0}, \bm{w}_{0}, p_{0}) \in \bm{X} $ be given. Additionally, let $ p_{0} \in H_{0}^{1}(\Omega) $  and $ \divv(2 \mu \bm{\epsilon}(\bm{u}_{0}) + \lambda \divv(\bm{u}_{0}) \bm{I} - \alpha  p_{0}  \bm{I}) \in \bm{L}^{2}(\Omega) $.
        Find $ (\bm{u}, \bm{v}, \bm{w}, p) \in L^{2}(I; \bm{X}) $ such that
        \begin{align}
            \bm{u}(\cdot, 0) & = \bm{u}_{0}, &
            \bm{v}(\cdot, 0) & = \bm{v}_{0}, &
            \bm{w}(\cdot, 0) & = \bm{w}_{0}, &
            p(\cdot, 0) & = p_{0},
        \end{align}
        and
        \begin{subequations}
            \begin{align}
                \int_{I} \scalar{\bar{\rho}  \partial_{t} \bm{v} + \rho_{f}  \partial_{t} \bm{w}}{\bm{\varphi}^{\bm{u}}}
                + a(\bm{u}, \bm{\varphi}^{\bm{u}})
                - \scalar{\alpha  p}{\divv(\bm{\varphi}^{\bm{u}})} \mathrm{d}t
                & = \int_{I} \scalar{\bm{f}}{\bm{\varphi}^{\bm{u}}} \mathrm{d}t, \\
                \int_{I} \scalar{\partial_{t} \bm{u} - \bm{v}}{\bm{\varphi}^{\bm{v}}} \mathrm{d}t
                & = 0, \\
                \int_{I} \scalar{\rho_{f}  \partial_{t} \bm{v} + \rho_{w}  \partial_{t} \bm{w} + \bm{K}^{-1} \bm{w}}{\bm{\varphi}^{\bm{w}}}
                - b(p, \bm{\varphi}^{\bm{w}}) \mathrm{d}t
                & = \int_{I} \scalar{\bm{g}}{\bm{\varphi}^{\bm{w}}} \mathrm{d}t, \\
                \int_{I} \scalar{s_{0}  \partial_{t} p + \alpha \divv(\partial_{t} \bm{u})}{\varphi^{p}}
                + b(\varphi^{p}, \bm{w}) \mathrm{d}t
                & = 0,
            \end{align}
        \end{subequations}
        for all $ (\bm{\varphi}^{\bm{u}}, \bm{\varphi}^{\bm{v}}, \bm{\varphi}^{\bm{w}}, \varphi^{p}) \in L^{2}(I; \bm{X}) $.
    \end{problem}
    The well-posedness of Problem~\ref{pr:dg_continuous_weak_formulation} is proven in \cite{Mielke2013Homo}.

    \subsection{Discontinuous Galerkin method}
    \label{subsec:dg_method}
    Here, we utilize a continuous Galerkin-Petrov (cGP($k$)) method in time of degree $k \in \mathbb{N}$ for problem~\eqref{eq:dg_dyn_biot_velo}.
    With this time stepping scheme, we end up with a globally continuous solution in time.
    For the spatial discretization, we exploit a discontinuous Galerkin method for the displacement $ \bm{u} $.
    To this end, we define the following discrete spaces
    \begin{align}
        \begin{aligned}
            \bm{U}_{h} = \bm{V}_{h} & := \{ \bm{u}_{h} \in \bm{H}(\divv, \Omega) : \bm{u}_{h}\big|_{K} \in \bm{U}(K) 
            \quad \forall K \in \mathcal{T}_{h} \text{ and } \bm{u} \cdot \bm{n} = 0 \text{ on $ \partial \Omega $} \}, \\
            \bm{W}_{h} & := \{ \bm{w}_{h} \in \bm{H}(\divv, \Omega) : \bm{w}_{h}\big|_{K} \in \bm{W}(K) 
            \quad \forall K \in \mathcal{T}_{h} \text{ and } \bm{w} \cdot \bm{n} = 0 \text{ on $ \partial \Omega $} \}, \\
            P_{h} & := \{ p_{h} \in L_{0}^{2}(\Omega) : p_{h}\big|_{K} \in P(K) \quad \forall K \in \mathcal{T}_{h} \}, \\
            \bm{X}_{h} & := \bm{U}_{h} \times \bm{V}_{h} \times \bm{W}_{h} \times P_{h},
        \end{aligned}
    \end{align}
    where $ \bm{U}(K) \times \bm{W}(K) \times P(K) $ here are chosen to be $ \text{BDM}_{\ell+1}(K) \times \text{BDM}_{\ell+1}(K) \times \mathbb{P}_{\ell}(K) $ for $ \ell \geq 0 $.
    The space $\text{BDM}_{\ell}(K)$ denotes the Brezzi-Douglas-Marini space of degree $ \ell \in \mathbb{N} $.
    
    In order to obtain mass conservation, we need $ \divv(\bm{U}_{h}) = P_{h} $ and $ \divv(\bm{W}_{h}) = P_{h} $, see \cite{Cockburn2007Anot}.
    For the discontinuous Galerkin method, we will partition the domain $\Omega$ by a shape-regular triangulation $ \mathcal{T}_{h} $.
    Let $ \average{\cdot} $ be the average and $ \llbracket \cdot \rrbracket $ the jump defined by
    \begin{align*}
        \average{\bm{\tau}} & := \frac{1}{2} \left( \bm{\tau}\big|_{\partial K_{1} \cap e} \bm{n}_{1} - \bm{\tau}\big|_{\partial K_{2} \cap e} \bm{n}_{2} \right), &
        \llbracket \bm{y} \rrbracket & := \bm{y}\big|_{\partial K_{1} \cap e} - \bm{y}\big|_{\partial K_{2} \cap e},
    \end{align*}
    for any interior edge $ e := K_{1} \cap K_{2} $ and
    \begin{align*}
        \average{\bm{\tau}} & := \bm{\tau}\big|_{e} \bm{n}, &
        \llbracket \bm{y} \rrbracket & := \bm{y}\big|_{e},
    \end{align*}
    for boundary edge $ e $, where $ \bm{\tau} $ denotes a matrix-valued and $ \bm{y} $ a vector-valued function, see e.g.~\cite{Hong2018parameter}.
    Additionally, we specify the tangential component by $ \tang(\bm{z}) := \bm{z} - (\bm{z} \cdot \bm{n}) \bm{n} $.
    Next, we introduce the following mesh-dependent bilinear forms
    \begin{align*}
        a_{h}(\bm{u}_{\tau, h}, \bm{\varphi}_{\tau, h}^{\bm{u}}) & := \sum_{K \in \mathcal{T}_{h}} 2 \mu \int_{K} \bm{\epsilon}(\bm{u}_{\tau, h}) : \bm{\epsilon}(\bm{\varphi}_{\tau, h}^{\bm{u}}) \mathrm{d}\bm{x}
        - \sum_{e \in \mathcal{E}_{h}} 2 \mu \int_{e} \average{\bm{\epsilon}(\bm{u}_{\tau, h})} \cdot \llbracket \tang(\bm{\varphi}_{\tau, h}^{\bm{u}}) \rrbracket \mathrm{d}s \\
        & \quad - \sum_{e \in \mathcal{E}_{h}} 2 \mu \int_{e} \average{\bm{\epsilon}(\bm{\varphi}_{\tau, h}^{\bm{u}})} \cdot \llbracket \tang(\bm{u}_{\tau, h}) \rrbracket \mathrm{d}s
        + \sum_{e \in \mathcal{E}_{h}} 2 \mu \eta h_{e}^{-1} \int_{e} \llbracket \tang(\bm{u}_{\tau, h}) \rrbracket \cdot \llbracket \tang(\bm{\varphi}_{\tau, h}^{\bm{u}}) \rrbracket \mathrm{d}s \\
        & \quad + \lambda \scalar{\divv(\bm{u}_{\tau, h})}{\divv(\bm{\varphi}_{\tau, h}^{\bm{u}})}
    \end{align*}
    and
    \begin{align*}
        b_{h}(p_{\tau, h}, \bm{\varphi}_{\tau, h}^{\bm{w}})
        := \scalar{p_{\tau, h}}{\divv(\bm{\varphi}_{\tau, h}^{\bm{w}})},
    \end{align*}
    where $ h_{e} $ is the size of the facet, $\mathcal{E}_{h}$ the set of all edges, $e$ the common boundary
    of two element in $\mathcal{T}_{h}$ and $\eta$ is a stabilization parameter, cf \cite{Hong2018parameter}.
    Then, using block vectors for all unknown fields, i.e.,
    $ \bm{x}_{\tau, h} := (\bm{u}_{\tau, h}^{\top}, \bm{v}_{\tau, h}^{\top}, \bm{w}_{\tau, h}^{\top}, p_{\tau, h})^{\top} $
    and approximation of initial values, i.e.,
    $ \bm{x}_{0, h} := (\bm{u}_{0, h}^{\top}, \bm{v}_{0, h}^{\top}, \bm{w}_{0, h}^{\top}, p_{0, h}) $ enables us to give a short representation of the discontinuous Galerkin method:
    
    \begin{problem}
        \label{pr:dg_discrete_weak_formulation}
        Let $ k, \ell \in \mathbb{N} $ and $ \bm{x}_{\tau, h}^{n-1} := \bm{x}_{\tau, h}(t_{n-1}) $ be given where $ \bm{x}_{\tau, h}(t_{0}) = \bm{x}_{0, h} $. Find $ \bm{x}_{\tau, h} \in \mathbb{P}_{k}(I_{n}; \bm{X}_{h}) $ such that $ \bm{x}_{\tau, h}(t_{n-1}) = \bm{x}_{\tau, h}^{n-1} $ and
        \begin{subequations}
            \label{eq:dg_dyn_biot_disc_weak_form}
            \begin{align}
                \int_{I_{n}} \scalar{\bar{\rho}  \partial_{t} \bm{v}_{\tau, h} + \rho_{f}  \partial_{t} \bm{w}_{\tau, h}}{\bm{\varphi}_{\tau, h}^{\bm{u}}}
                + a_{h}(\bm{u}_{\tau, h}, \bm{\varphi}_{\tau, h}^{\bm{u}})
                - \scalar{\alpha  p_{\tau, h}}{\divv(\bm{\varphi}_{\tau, h}^{\bm{u}})} \mathrm{d}t
                & = \int_{I_{n}} \scalar{\bm{\mathcal{I}}_{\tau}^{*} \bm{f}}{\bm{\varphi}_{\tau, h}^{\bm{u}}} \mathrm{d}t, \label{eq:dg_dyn_biot_disc_weak_form-1} \\
                \int_{I_{n}} \scalar{\partial_{t} \bm{u}_{\tau, h} - \bm{v}_{\tau, h}}{\bm{\varphi}_{\tau, h}^{\bm{v}}} \mathrm{d}t
                & = 0, \label{eq:dg_dyn_biot_disc_weak_form-2} \\
                \int_{I_{n}} \scalar{\rho_{f}  \partial_{t} \bm{v}_{\tau, h} + \rho_{w}  \partial_{t} \bm{w}_{\tau, h} + \bm{K}^{-1} \bm{w}_{\tau, h}}{\bm{\varphi}_{\tau, h}^{\bm{w}}}
                - b_{h}(p_{\tau, h}, \bm{\varphi}_{\tau, h}^{\bm{w}}) \mathrm{d}t
                & = \int_{I_{n}} \scalar{\bm{\mathcal{I}}_{\tau}^{*} g}{\bm{\varphi}_{\tau, h}^{\bm{w}}} \mathrm{d}t, \label{eq:dg_dyn_biot_disc_weak_form-3} \\
                \int_{I_{n}} \scalar{s_{0}  \partial_{t} p_{\tau, h} + \alpha \divv(\partial_{t} \bm{u}_{\tau, h})}{\varphi_{\tau, h}^{p}}
                + b_{h}(\varphi_{\tau, h}^{p}, \bm{w}_{\tau, h}) \mathrm{d}t
                & = 0, \label{eq:dg_dyn_biot_disc_weak_form-4}
            \end{align}
        \end{subequations}
        for all $ (\bm{\varphi}_{\tau, h}^{\bm{u}}, \bm{\varphi}_{\tau, h}^{\bm{v}}, \bm{\varphi}_{\tau, h}^{\bm{w}}, \varphi_{\tau, h}^{p}) \in \mathbb{P}_{k-1}(I_{n}; \bm{X}_{h}) $.
    \end{problem}
    Recall that $ \mathcal{I}_{\tau}^{*} $ denotes any interpolation operator that locally corresponds to $ k+1 $ points on $I_{n}$, $ n=1, \ldots,N$, as described in Section~\ref{subsec:quadrature_formulas}.

    \section{Error analysis}\label{sec:5}
    In this section, we derive an error estimate for the solution of Problem~\ref{pr:dg_discrete_weak_formulation}.
    For this purpose, we define the following norms
    \begin{subequations}
        \label{eq:dg_definition_DG_norm}
        \begin{align}
            \norm{\bm{u}_{h}}_{\text{DG}}^{2} & := \sum_{K \in \mathcal{T}_{h}} \norm{\nabla \bm{u}_{h}}_{L^{2}(K)}^{2}
            + \sum_{e \in \mathcal{E}_{h}} h_{e}^{-1} \norm{\llbracket \tang(\bm{u}_{h}) \rrbracket}_{L^{2}(e)}^{2}
            + \sum_{K \in \mathcal{T}_{h}} h_{K}^{2}  \lvert \bm{u}_{h} \rvert_{H^{2}(K)}^{2}, \\
            \norm{\bm{u}_{h}}_{\bm{U}_{h}}^{2} & := \norm{\bm{u}_{h}}_{\text{DG}}^{2} + \norm{\divv(\bm{u}_{h})}_{L^{2}(\Omega)}^{2},
        \end{align}
    \end{subequations}
    where $ h_{K} $ represents the mesh sizes of cell $ K $.

    Next, we deduce a formulation of \eqref{eq:dg_dyn_biot_velo} that is suitable for our analysis.
    Therefore, we evaluate the exact solution of \eqref{eq:dg_dyn_biot_velo} in the discrete time points $ t_{n, i}^{\text{GL}},  i = 0, \ldots, k $ on subinterval $ I_{n} $, multiply then by the corresponding Lagrange polynomial $ L_{n, i}^{\text{GL}} $ and sum over the resulting $ k + 1 $ equations, respectively. After that, we obtain the following system
    \begin{subequations}
        \begin{align}
            \bar{\rho}  \bm{\mathcal{I}}_{\tau}^{\text{GL}}(\partial_{t} \bm{v}) - 2 \mu \divv(\bm{\epsilon}(\bm{\mathcal{I}}_{\tau}^{\text{GL}} \bm{u})) - \lambda  \nabla \divv(\bm{\mathcal{I}}_{\tau}^{\text{GL}} \bm{u}) + \alpha  \nabla \mathcal{I}_{\tau}^{\text{GL}} p + \rho_{f}  \bm{\mathcal{I}}_{\tau}^{\text{GL}}(\partial_{t} \bm{w}) & = \bm{\mathcal{I}}_{\tau}^{\text{GL}} \bm{f}, \\
            \bm{\mathcal{I}}_{\tau}^{\text{GL}}(\partial_{t} \bm{u}) - \bm{\mathcal{I}}_{\tau}^{\text{GL}} \bm{v} & = \bm{0}, \\
            \rho_{f}  \bm{\mathcal{I}}_{\tau}^{\text{GL}}(\partial_{t} \bm{v}) + \rho_{w}  \bm{\mathcal{I}}_{\tau}^{\text{GL}}(\partial_{t} \bm{w}) + \bm{K}^{-1} \bm{\mathcal{I}}_{\tau}^{\text{GL}} \bm{w} + \nabla \mathcal{I}_{\tau}^{\text{GL}} p & = \bm{\mathcal{I}}_{\tau}^{\text{GL}} \bm{g}, \\
            s_{0}  \mathcal{I}_{\tau}^{\text{GL}}(\partial_{t} p) + \alpha \divv(\bm{\mathcal{I}}_{\tau}^{\text{GL}}(\partial_{t} \bm{u})) + \divv(\bm{\mathcal{I}}_{\tau}^{\text{GL}} \bm{w}) & = 0.
        \end{align}
    \end{subequations}
    Now we introduce the residual of a function $ y \in C^{1}(\overline{I}_{n}; B) $ defined by
    \begin{align}
        \label{eq:dg_definition_r_y}
        r_{y} := \partial_{t} \mathcal{I}_{\tau}^{\text{GL}} y - \mathcal{I}_{\tau}^{\text{GL}}(\partial_{t} y).
    \end{align}
    With this, we can reformulate the system above as
    \begin{subequations}
        \label{eq:dg_dyn_biot_interpolation_form}
        \begin{align}
            \bar{\rho}  \partial_{t} \bm{\mathcal{I}}_{\tau}^{\text{GL}} \bm{v} - 2 \mu \divv(\bm{\epsilon}(\bm{\mathcal{I}}_{\tau}^{\text{GL}} \bm{u})) - \lambda  \nabla \divv(\bm{\mathcal{I}}_{\tau}^{\text{GL}} \bm{u}) + \alpha  \nabla \mathcal{I}_{\tau}^{\text{GL}} p + \rho_{f}  \partial_{t} \bm{\mathcal{I}}_{\tau}^{\text{GL}} \bm{w} & = \bm{\mathcal{I}}_{\tau}^{\text{GL}} \bm{f} + \bar{\rho}  \bm{r}_{\bm{v}} + \rho_{f}  \bm{r}_{\bm{w}}, \label{eq:dg_dyn_biot_interpolation_form-1} \\
            \partial_{t} \bm{\mathcal{I}}_{\tau}^{\text{GL}} \bm{u} - \bm{\mathcal{I}}_{\tau}^{\text{GL}} \bm{v} & = \bm{r}_{\bm{u}}, \label{eq:dg_dyn_biot_interpolation_form-2} \\
            \rho_{f}  \partial_{t} \bm{\mathcal{I}}_{\tau}^{\text{GL}} \bm{v} + \rho_{w}  \partial_{t} \bm{\mathcal{I}}_{\tau}^{\text{GL}} \bm{w} + \bm{K}^{-1} \bm{\mathcal{I}}_{\tau}^{\text{GL}} \bm{w} + \nabla \mathcal{I}_{\tau}^{\text{GL}} p & = \bm{\mathcal{I}}_{\tau}^{\text{GL}} \bm{g} + \rho_{f}  \bm{r}_{\bm{v}} + \rho_{w}  \bm{r}_{\bm{w}}, \label{eq:dg_dyn_biot_interpolation_form-3} \\
            s_{0}  \partial_{t} \mathcal{I}_{\tau}^{\text{GL}} p + \alpha \divv(\partial_{t} \bm{\mathcal{I}}_{\tau}^{\text{GL}} \bm{u}) + \divv(\bm{\mathcal{I}}_{\tau}^{\text{GL}} \bm{w}) & = s_{0}  r_{p} + \alpha \divv(\bm{r}_{\bm{u}}). \label{eq:dg_dyn_biot_interpolation_form-4}
        \end{align}
    \end{subequations}
    We define the following projection operators
    \begin{subequations}
        \label{eq:dg_projection_operators}
        \begin{alignat}{3}
            \bm{\mathcal{P}}_{1} & : \bm{U} \to \bm{U}_{h}, \quad &
            a_{h}(\bm{\mathcal{P}_{1}} \bm{y}, \bm{\varphi}_{h}^{\bm{u}}) & = a_{h}(\bm{y}, \bm{\varphi}_{h}^{\bm{u}}) & \qquad & \forall \bm{\varphi}_{h}^{\bm{u}} \in \bm{U}_{h}, \label{eq:dg_projection_operators-1} \\
            \bm{\mathcal{P}}_{2} & : \bm{W} \to \bm{W}_{h}, \quad &
            \scalar{\divv(\bm{\mathcal{P}}_{2} \bm{y})}{\varphi_{h}^{p}} & = \scalar{\divv(\bm{y})}{\varphi_{h}^{p}} & \qquad & \forall \varphi_{h}^{p} \in P_{h}, \label{eq:dg_projection_operators-2} \\
            \mathcal{P}_{3} & : P \to P_{h}, \quad &
            \scalar{\mathcal{P}_{3} y}{\varphi_{h}^{p}} & = \scalar{y}{\varphi_{h}^{p}} & \qquad & \forall \varphi_{h}^{p} \in P_{h}, \label{eq:dg_projection_operators-3}
        \end{alignat}
    \end{subequations}
    that are used for our splitting of the discretization error below.
    Some error estimates for these projection operators are summarized in the following, see \cite[Sec.~4.3]{Arnold2001Unif}, \cite[Prop.~2.5.4]{Boffi2013Mixe}, \cite[Cor.~1.109]{Ern2004Theo}:
    \begin{lemma}[Projection errors]
        \label{lem:dg_projection_errors}
        The projection estimates
        \begin{align*}
            \norm{\bm{u} - \bm{\mathcal{P}}_{1} \bm{u}}_{\bm{L}^{2}(\Omega)}
            & \leq c h^{\ell+1}  \norm{\bm{u}}_{\bm{H}^{\ell+1}(\Omega)}, &
            \norm{\divv(\bm{u} - \bm{\mathcal{P}}_{1} \bm{u})}_{L^{2}(\Omega)}
            & \leq c h^{\ell+1}  \norm{\bm{u}}_{\bm{H}^{\ell+2}(\Omega)}, \\
            \norm{\bm{u} - \bm{\mathcal{P}}_{1} \bm{u}}_{\textnormal{DG}}
            & \leq c h^{\ell+1}  \norm{\bm{u}}_{\bm{H}^{\ell+2}(\Omega)}, &
            \norm{p - \mathcal{P}_{3} p}_{L^{2}(\Omega)}
            & \leq c h^{\ell+1}  \norm{p}_{H^{\ell+1}(\Omega)}, \\
            \norm{\bm{w} - \bm{\mathcal{P}}_{2} \bm{w}}_{\bm{L}^{2}(\Omega)}
            & \leq c h^{\ell+1}  \norm{\bm{w}}_{\bm{H}^{\ell+1}(\Omega)}, &
            \norm{\divv(\bm{w} - \bm{\mathcal{P}}_{2} \bm{w})}_{L^{2}(\Omega)}
            & \leq c h^{\ell+1}  \norm{\bm{w}}_{\bm{H}^{\ell+2}(\Omega)},
        \end{align*}
        are fulfilled for $ \bm{u} \in \bm{H}^{\ell+2}(\Omega) $, $ \bm{w} \in \bm{H}^{\ell+2}(\Omega) $ and $ p \in H^{\ell+1}(\Omega) $.
    \end{lemma}
    
    Next, we split the discretization errors into
    \begin{subequations}
        \label{eq:dg_splitting}
        \begin{align}
            \bm{u}_{\tau, h} - \bm{u}
            & = (\bm{u}_{\tau, h} - \bm{z}^{\bm{u}})
            + (\bm{z}^{\bm{u}} - \bm{\mathcal{I}}_{\tau}^{\text{GL}} \bm{u})
            + (\bm{\mathcal{I}}_{\tau}^{\text{GL}} \bm{u} - \bm{u})
            =: \bm{e}_{\bm{u}} - \bm{\eta}_{\bm{u}} - \bm{\eta}_{\bm{u}, \tau}, \label{eq:dg_splitting-1} \\
            \bm{v}_{\tau, h} - \bm{v}
            & = (\bm{v}_{\tau, h} - \bm{z}^{\bm{v}})
            + (\bm{z}^{\bm{v}} - \bm{\mathcal{I}}_{\tau}^{\text{GL}} \bm{v})
            + (\bm{\mathcal{I}}_{\tau}^{\text{GL}} \bm{v} - \bm{v})
            =: \bm{e}_{\bm{v}} - \bm{\eta}_{\bm{v}} - \bm{\eta}_{\bm{v}, \tau}, \label{eq:dg_splitting-2} \\
            \bm{w}_{\tau, h} - \bm{w}
            & = (\bm{w}_{\tau, h} - \bm{z}^{\bm{w}})
            + (\bm{z}^{\bm{w}} - \bm{\mathcal{I}}_{\tau}^{\text{GL}} \bm{w})
            + (\bm{\mathcal{I}}_{\tau}^{\text{GL}} \bm{w} - \bm{w})
            =: \bm{e}_{\bm{w}} - \bm{\eta}_{\bm{w}} - \bm{\eta}_{\bm{w}, \tau}, \label{eq:dg_splitting-3} \\
            p_{\tau, h} - p
            & = (p_{\tau, h} - z^{p})
            + (z^{p} - \mathcal{I}_{\tau}^{\text{GL}} p)
            + (\mathcal{I}_{\tau}^{\text{GL}} p - p)
            =: e_{p} - \eta_{p} - \eta_{p, \tau}, \label{eq:dg_splitting-4}
        \end{align}
    \end{subequations}
    where
    \begin{subequations}
        \label{eq:dg_definition_z}
        \begin{align}
            \bm{z}^{\bm{u}} & := \bm{\mathcal{I}}_{\tau}^{\text{GL}}\left( \int_{t_{n-1}}^{t} \bm{\mathcal{I}}_{\tau}^{\text{GL}} \bm{\mathcal{P}}_{1} \partial_{t} \bm{u}(s) \mathrm{d}s + \bm{\mathcal{P}}_{1} \bm{u}(t_{n-1}) \right) \in \mathbb{P}_{k}(I_{n}; \bm{U}_{h}), \\
            \bm{z}^{\bm{v}} & := \bm{\mathcal{I}}_{\tau}^{\text{GL}}\left( \int_{t_{n-1}}^{t} \bm{\mathcal{I}}_{\tau}^{\text{GL}} \bm{\mathcal{P}}_{1} \partial_{t} \bm{v}(s) \mathrm{d}s + \bm{\mathcal{P}}_{1} \bm{v}(t_{n-1}) \right) \in \mathbb{P}_{k}(I_{n}; \bm{V}_{h}), \\
            \bm{z}^{\bm{w}} & := \bm{\mathcal{I}}_{\tau}^{\text{GL}}\left( \int_{t_{n-1}}^{t} \bm{\mathcal{I}}_{\tau}^{\text{GL}} \bm{\mathcal{P}}_{2} \partial_{t} \bm{w}(s) \mathrm{d}s + \bm{\mathcal{P}}_{2} \bm{w}(t_{n-1}) \right) \in \mathbb{P}_{k}(I_{n}; \bm{W}_{h}), \\
            z^{p} & := \mathcal{I}_{\tau}^{\text{GL}}\left( \int_{t_{n-1}}^{t} \mathcal{I}_{\tau}^{\text{GL}} \mathcal{P}_{3} \partial_{t} p(s) \mathrm{d}s + \mathcal{P}_{3} p(t_{n-1}) \right) \in \mathbb{P}_{k}(I_{n}; P_{h}),
        \end{align}
    \end{subequations}
    using the interpolation operator $ \mathcal{I}_{\tau}^{\text{GL}} $ from Section~\ref{subsec:quadrature_formulas}.
    Note
    that the last term in each equation in \eqref{eq:dg_splitting} can be estimated by \eqref{eq:interpolation_error_GL-1}.
    The expressions in \eqref{eq:dg_definition_z} possess the following properties:
    \begin{lemma}
        \label{lem:dg_identity_z}
        The following identities
        \begin{alignat*}{2}
            \int_{I_{n}} \scalar{\partial_{t} \bm{z}^{\bm{u}} - \bm{\mathcal{I}}_{\tau}^{\textnormal{GL}} \bm{\mathcal{P}}_{1} \partial_{t} \bm{u}}{\bm{\varphi}} \mathrm{d}t
            & = 0 & \qquad & \forall \bm{\varphi} \in \mathbb{P}_{k-1}(I_{n}; \bm{L}^{2}(\Omega)), \\
            \int_{I_{n}} \scalar{\divv(\partial_{t} \bm{z}^{\bm{u}} - \bm{\mathcal{I}}_{\tau}^{\textnormal{GL}} \bm{\mathcal{P}}_{1} \partial_{t} \bm{u})}{\bm{\varphi}} \mathrm{d}t
            & = 0 & \qquad & \forall \bm{\varphi} \in \mathbb{P}_{k-1}(I_{n}; \bm{L}^{2}(\Omega)), \\
            \int_{I_{n}} \scalar{\partial_{t} \bm{z}^{\bm{v}} - \bm{\mathcal{I}}_{\tau}^{\textnormal{GL}} \bm{\mathcal{P}}_{1} \partial_{t} \bm{v}}{\bm{\varphi}} \mathrm{d}t
            & = 0 & \qquad & \forall \bm{\varphi} \in \mathbb{P}_{k-1}(I_{n}; \bm{L}^{2}(\Omega)), \\
            \int_{I_{n}} \scalar{\partial_{t} \bm{z}^{\bm{w}} - \bm{\mathcal{I}}_{\tau}^{\textnormal{GL}} \bm{\mathcal{P}}_{2} \partial_{t} \bm{w}}{\bm{\varphi}} \mathrm{d}t
            & = 0 & \qquad & \forall \bm{\varphi} \in \mathbb{P}_{k-1}(I_{n}; \bm{L}^{2}(\Omega)), \\
            \int_{I_{n}} \scalar{\partial_{t} z^{p} - \mathcal{I}_{\tau}^{\textnormal{GL}} \mathcal{P}_{3} \partial_{t} p}{\varphi} \mathrm{d}t
            & = 0 & \qquad & \forall \varphi \in \mathbb{P}_{k-1}(I_{n}; L^{2}(\Omega)),
        \end{alignat*}
        hold true.
    \end{lemma}
    
    \begin{proof}
        The proof is analogous to \cite[Lem.~3.1]{Karakashian2004Conv}.
    \end{proof}
    
    To further progress with the preparatory steps of our analysis, we introduce suitable norms for the error estimation.
    Collecting all norms participating in the upper bound, we define the norms
    \begin{align*}
        \norm{\bm{u}}_{\bm{u}, n}^{2}
        & := \norm{\partial_{t}^{k+3} \bm{u}}_{L^{2}(I_{n}; \bm{H}^{\ell+1})}^{2}
        + \norm{\partial_{t}^{k+2} \bm{u}}_{L^{2}(I_{n}; \bm{H}^{\ell+2})}^{2}
        + \norm{\partial_{t}^{k+1} \bm{u}}_{L^{2}(I_{n}; \bm{H}^{\ell+2})}^{2}
        + \norm{\partial_{t}^{2} \bm{u}}_{L^{2}(I_{n}; \bm{H}^{\ell+1})}^{2} \\
        & \quad + \norm{\partial_{t} \bm{u}}_{L^{2}(I_{n}; \bm{H}^{\ell+2})}^{2}
        + \norm{\bm{u}}_{L^{2}(I_{n}; \bm{H}^{\ell+2})}^{2}
        + \norm{\partial_{t}^{k+3} \bm{u}}_{L^{\infty}(I_{n}; \bm{H}^{\ell+1})}^{2}
        + \norm{\partial_{t}^{k+2} \bm{u}}_{L^{\infty}(I_{n}; \bm{H}^{\ell+2})}^{2} \\
        & \quad + \norm{\partial_{t}^{k+1} \bm{u}}_{L^{\infty}(I_{n}; \bm{H}^{\ell+2})}^{2}
        + \norm{\partial_{t} \bm{u}}_{L^{\infty}(I_{n}; \bm{H}^{\ell+1})}^{2}
        + \norm{\bm{u}}_{L^{\infty}(I_{n}; \bm{H}^{\ell+2})}^{2}, \\
        \norm{\bm{w}}_{\bm{w}, n}^{2}
        & := \norm{\partial_{t}^{k+2} \bm{w}}_{L^{2}(I_{n}; \bm{H}^{\ell+2})}^{2}
        + \norm{\partial_{t}^{k+1} \bm{w}}_{L^{2}(I_{n}; \bm{H}^{\ell+1})}^{2}
        + \norm{\partial_{t} \bm{w}}_{L^{2}(I_{n}; \bm{H}^{\ell+1})}^{2}
        + \norm{\bm{w}}_{L^{2}(I_{n}; \bm{H}^{\ell+1})}^{2} \\
        & \quad + \norm{\partial_{t}^{k+2} \bm{w}}_{L^{\infty}(I_{n}; \bm{H}^{\ell+1})}^{2}
        + \norm{\partial_{t}^{k+1} \bm{w}}_{L^{\infty}(I_{n}; \bm{H}^{\ell+1})}^{2}
        + \norm{\bm{w}}_{L^{\infty}(I_{n}; \bm{H}^{\ell+1})}^{2}, \\
        \norm{p}_{p, n}^{2}
        & := \norm{\partial_{t}^{k+2} p}_{L^{2}(I_{n}; H^{\ell+1})}^{2}
        + \norm{\partial_{t}^{k+1} p}_{L^{2}(I_{n}; H^{\ell+1})}^{2}
        + \norm{p}_{L^{2}(I_{n}; H^{\ell+1})}^{2} \\
        & \quad + \norm{\partial_{t}^{k+2} p}_{L^{\infty}(I_{n}; H^{\ell+1})}^{2}
        + \norm{\partial_{t}^{k+1} p}_{L^{\infty}(I_{n}; H^{\ell+1})}^{2}
        + \norm{p}_{L^{\infty}(I_{n}; H^{\ell+1})}^{2},
    \end{align*}
    and, finally, the corresponding combined norms
    \begin{align}
        \label{eq:dg_definition_norm_n}
        \norm{(\bm{u}, \bm{w}, p)}_{n}^{2} & :=
        \norm{\bm{u}}_{\bm{u}, n}^{2}
        + \norm{\bm{w}}_{\bm{w}, n}^{2}
        + \norm{p}_{p, n}^{2}, &
        \norm{(\bm{u}, \bm{w}, p)}_{I}^{2} & := \sum_{n=1}^{N} \norm{(\bm{u}, \bm{w}, p)}_{n}^{2}.
    \end{align}
    
    With these definitions at hand, the following estimates can be proven:
    
    \begin{lemma}
        \label{lem:dg_estimates_eta}
        There holds
        \begin{align}
            \norm{\bm{\eta}_{\bm{u}}}_{L^{2}(I_{n}; \bm{U}_{h})}
            & \leq c \left( \tau^{k+2} +h^{\ell+1}\right) \norm{(\bm{u}, \bm{w}, p)}_{n}, \label{eq:dg_estimate_eta-1} \\
            \norm{\bm{\eta}_{\bm{v}}}_{L^{2}(I_{n}; \bm{L}^{2})}
            & \leq c \left( \tau^{k+2} +h^{\ell+1} \right) \norm{(\bm{u}, \bm{w}, p)}_{n}, \label{eq:dg_estimate_eta-2} \\
            \norm{\bm{\eta}_{\bm{w}}}_{L^{2}(I_{n}; \bm{L}^{2})}
            & \leq c \left( \tau^{k+2} +h^{\ell+1} \right) \norm{(\bm{u}, \bm{w}, p)}_{n}, \label{eq:dg_estimate_eta-3} \\
            \norm{\eta_{p}}_{L^{2}(I_{n}; L^{2})}
            & \leq c \left( \tau^{k+2} +h^{\ell+1} \right) \norm{(\bm{u}, \bm{w}, p)}_{n}. \label{eq:dg_estimate_eta-4}
        \end{align}
    \end{lemma}
    
    \begin{proof}
        Exploiting Lemma~\ref{lem:dg_projection_errors} and \eqref{eq:interpolation_error_GL-2} gives
        \begin{align}
            \label{eq:prf_dg_estimates_eta-1}
            \norm{\bm{\mathcal{I}}_{\tau}^{\text{GL}} \bm{w} - \bm{\mathcal{I}}_{\tau}^{\text{GL}} \bm{\mathcal{P}}_{2} \bm{w}}_{L^{2}(I_{n}; \bm{L}^{2})}
            & \leq c h^{\ell+1} \left( \tau^{k+1}  \norm{\partial_{t}^{k+1} \bm{w}}_{L^{2}(I_{n}; \bm{H}^{\ell+1})}
            + \norm{\bm{w}}_{L^{2}(I_{n}; \bm{H}^{\ell+1})} \right).
        \end{align}
        We recall that 
        \begin{align*}
            \norm{\mathcal{I}_{\tau}^{\text{GL}} f}_{L^{2}(I_{n}; L^{2})} & \leq c \norm{f}_{L^{2}(I_{n}; L^{2})} + c \tau \norm{\partial_{t} f}_{L^{2}(I_{n}; L^{2})}, \\
            \norm{\int_{t_{n-1}}^{t} f \mathrm{d}s}_{L^{2}(I_{n}; L^{2})} & \leq c \tau \norm{f}_{L^{2}(I_{n}; L^{2})},
        \end{align*}
        are valid, see \cite[Eqs.~(3.15)--(3.16)]{Karakashian2004Conv}.
        Using these inequalities coupled with \eqref{eq:interpolation_error_GL-1} and Lemma~\ref{lem:dg_projection_errors} leads to
        \begin{align}
            \label{eq:prf_dg_estimates_eta-2}
            \begin{aligned}
                \norm{\bm{\mathcal{I}}_{\tau}^{\text{GL}} \bm{\mathcal{P}}_{2} \bm{w} - \bm{z}^{\bm{w}}}_{L^{2}(I_{n}; \bm{L}^{2})}
                & = \norm{\bm{\mathcal{I}}_{\tau}^{\text{GL}}\left( \int_{t_{n-1}}^{t} \bm{\mathcal{I}}_{\tau}^{\text{GL}} \bm{\mathcal{P}}_{2} \partial_{t} \bm{w} - \bm{\mathcal{P}}_{2} \partial_{t} \bm{w} \mathrm{d}s \right)}_{L^{2}(I_{n}; \bm{L}^{2})} \\
                & \leq c \tau \norm{\bm{\mathcal{I}}_{\tau}^{\text{GL}} \bm{\mathcal{P}}_{2} \partial_{t} \bm{w} - \bm{\mathcal{P}}_{2} \partial_{t} \bm{w}}_{L^{2}(I_{n}; \bm{L}^{2})} \\
                & \leq c \tau^{k+2} \left( h^{\ell+1} \norm{\partial_{t}^{k+2} \bm{w}}_{L^{2}(I_{n}; \bm{H}^{\ell+1})}
                + \norm{\partial_{t}^{k+2} \bm{w}}_{L^{2}(I_{n}; \bm{L}^{2})} \right).
            \end{aligned}
        \end{align}
        Now the estimates \eqref{eq:prf_dg_estimates_eta-1}--\eqref{eq:prf_dg_estimates_eta-2} prove
        \begin{align*}
            \norm{\bm{\eta}_{\bm{w}}}_{L^{2}(I_{n}; \bm{L}^{2})}
            & \leq \norm{\bm{\mathcal{I}}_{\tau}^{\text{GL}} \bm{w} - \bm{\mathcal{I}}_{\tau}^{\text{GL}} \bm{\mathcal{P}}_{2} \bm{w}}_{L^{2}(I_{n}; \bm{L}^{2})}
            + \norm{\bm{\mathcal{I}}_{\tau}^{\text{GL}} \bm{\mathcal{P}}_{2} \bm{w} - \bm{z}^{\bm{w}}}_{L^{2}(I_{n}; \bm{L}^{2})}
            \leq c \left( \tau^{k+2} + h^{\ell+1} \right) \norm{(\bm{u}, \bm{w}, p)}_{n}.
        \end{align*}
        Similarly, we can show \eqref{eq:dg_estimate_eta-2}--\eqref{eq:dg_estimate_eta-4}.
    \end{proof}
    
    In order to simplify the notation, we define the symmetric positive definite matrix
    \begin{align*}
        \bm{M}_{\rho} :=
        \begin{pmatrix}
            \bar{\rho}  \bm{I} & \rho_{f}  \bm{I} \\
            \rho_{f}  \bm{I} & \rho_{w}  \bm{I}
        \end{pmatrix}.
    \end{align*}
    Next, we integrate \eqref{eq:dg_dyn_biot_interpolation_form} over space and time, apply integration by parts and subtract the resulting system from equation \eqref{eq:dg_dyn_biot_disc_weak_form}.
    After adding the first and third equation of the system, we obtain
    \begin{subequations}
        \label{eq:dg_dyn_biot_var_eq}
        \begin{align}
            \begin{aligned}
                & \quad \int_{I_{n}} \scalar{\bm{M}_{\rho}
                    \begin{pmatrix}
                        \partial_{t} \bm{e}_{\bm{v}} \\
                        \partial_{t} \bm{e}_{\bm{w}}
                    \end{pmatrix}
                }{
                    \begin{pmatrix}
                        \bm{\varphi}_{\tau, h}^{\bm{u}} \\
                        \bm{\varphi}_{\tau, h}^{\bm{w}}
                    \end{pmatrix}
                }
                + a_{h}(\bm{e}_{\bm{u}}, \bm{\varphi}_{\tau, h}^{\bm{u}})
                + \scalar{\bm{K}^{-1} \bm{e}_{\bm{w}}}{\bm{\varphi}_{\tau, h}^{\bm{w}}}
                - \scalar{\alpha  e_{p}}{\divv(\bm{\varphi}_{\tau, h}^{\bm{u}})}
                - \scalar{e_{p}}{\divv(\bm{\varphi}_{\tau, h}^{\bm{w}})} \mathrm{d}t \\
                & = \int_{I_{n}} \scalar{\bm{M}_{\rho}
                    \begin{pmatrix}
                        \partial_{t} \bm{\eta}_{\bm{v}} \\
                        \partial_{t} \bm{\eta}_{\bm{w}}
                    \end{pmatrix}
                }{
                    \begin{pmatrix}
                        \bm{\varphi}_{\tau, h}^{\bm{u}} \\
                        \bm{\varphi}_{\tau, h}^{\bm{w}}
                    \end{pmatrix}
                }
                + a_{h}(\bm{\eta}_{\bm{u}}, \bm{\varphi}_{\tau, h}^{\bm{u}})
                + \scalar{\bm{K}^{-1} \bm{\eta}_{\bm{w}}}{\bm{\varphi}_{\tau, h}^{\bm{w}}}
                - \scalar{\alpha  \eta_{p}}{\divv(\bm{\varphi}_{\tau, h}^{\bm{u}})}
                - \scalar{\eta_{p}}{\divv(\bm{\varphi}_{\tau, h}^{\bm{w}})} \mathrm{d}t \\
                & \quad - \int_{I_{n}} \scalar{\bm{M}_{\rho}
                    \begin{pmatrix}
                        \bm{r}_{\bm{v}} \\
                        \bm{r}_{\bm{w}}
                    \end{pmatrix}
                }{
                    \begin{pmatrix}
                        \bm{\varphi}_{\tau, h}^{\bm{u}} \\
                        \bm{\varphi}_{\tau, h}^{\bm{w}}
                    \end{pmatrix}
                }
                + \int_{I_{n}} \scalar{
                    \begin{pmatrix}
                        \bm{\mathcal{I}}_{\tau}^{*} \bm{f} \\
                        \bm{\mathcal{I}}_{\tau}^{*} \bm{g}
                    \end{pmatrix}
                    -
                    \begin{pmatrix}
                        \bm{\mathcal{I}}_{\tau}^{\text{GL}} \bm{f} \\
                        \bm{\mathcal{I}}_{\tau}^{\text{GL}} \bm{g}
                    \end{pmatrix}
                }{
                    \begin{pmatrix}
                        \bm{\varphi}_{\tau, h}^{\bm{u}} \\
                        \bm{\varphi}_{\tau, h}^{\bm{w}}
                    \end{pmatrix}
                }
                \mathrm{d}t
            \end{aligned}
            \label{eq:dg_dyn_biot_var_eq-1}
        \end{align}
        where we have used the consistency of $ a_{h}(\cdot, \cdot) $ and our splitting~\eqref{eq:dg_splitting}.
        The second equation becomes
        \begin{align}
            \int_{I_{n}} \scalar{\partial_{t} \bm{e}_{\bm{u}} - \bm{e}_{\bm{v}}}{\bm{\varphi}_{\tau, h}^{\bm{v}}} \mathrm{d}t
            & = \int_{I_{n}} \scalar{\partial_{t} \bm{\eta}_{\bm{u}} - \bm{\eta}_{\bm{v}} - \bm{r}_{\bm{u}}}{\bm{\varphi}_{\tau, h}^{\bm{v}}} \mathrm{d}t
            \label{eq:dg_dyn_biot_var_eq-2}
        \end{align}
        and finally, we get
        \begin{align}
            \begin{aligned}
                & \quad \int_{I_{n}} \scalar{s_{0}  \partial_{t} e_{p} + \alpha \divv(\partial_{t} \bm{e}_{\bm{u}}) + \divv(\bm{e}_{\bm{w}})}{\varphi_{\tau, h}^{p}} \mathrm{d}t \\
                & = \int_{I_{n}} \scalar{s_{0}  \partial_{t} \eta_{p} + \alpha \divv(\partial_{t} \bm{\eta}_{\bm{u}}) + \divv(\bm{\eta}_{\bm{w}})}{\varphi_{\tau, h}^{p}} \mathrm{d}t
                - \int_{I_{n}} \scalar{s_{0}  r_{p} + \alpha \divv(\bm{r}_{\bm{u}})}{\varphi_{\tau, h}^{p}} \mathrm{d}t \\
                & = \int_{I_{n}} \scalar{\alpha \divv(\partial_{t} \bm{\eta}_{\bm{u}}) + \divv(\bm{\eta}_{\bm{w}})}{\varphi_{\tau, h}^{p}} \mathrm{d}t
                - \int_{I_{n}} \scalar{\alpha \divv(\bm{r}_{\bm{u}})}{\varphi_{\tau, h}^{p}} \mathrm{d}t,
                \label{eq:dg_dyn_biot_var_eq-3}
            \end{aligned}
        \end{align}
    \end{subequations}
    where in the last step we have used
    \begin{align}
        \label{eq:dg_vanishing_term}
        \begin{aligned}
            \int_{I_{n}} s_{0} \scalar{\partial_{t} \eta_{p} - r_{p}}{\varphi_{\tau, h}^{p}} \mathrm{d}t
            & = \int_{I_{n}} s_{0} \scalar{\mathcal{I}_{\tau}^{\text{GL}} \partial_{t} p - \mathcal{P}_{3} \mathcal{I}_{\tau}^{\text{GL}} \partial_{t} p}{\varphi_{\tau, h}^{p}} \mathrm{d}t
            = 0
            \qquad \forall \varphi_{\tau, h}^{p} \in \mathbb{P}_{k-1}(I_{n}; P_{h}).
        \end{aligned}
    \end{align}
    
    In the next two lemmas, we provide estimates for the right-hand sides of \eqref{eq:dg_dyn_biot_var_eq}:
    
    \begin{lemma}[Estimates for \eqref{eq:dg_dyn_biot_var_eq-1}]
        \label{lem:dg_aux_estimates_rhs_I}
        The following estimates
        \begin{subequations}
            \label{eq:dg_aux_estimates_rhs_I}
            \begin{align}
                \int_{I_{n}} \scalar{\bm{M}_{\rho}
                    \begin{pmatrix}
                        \partial_{t} \bm{\eta}_{\bm{v}} - \bm{r}_{\bm{v}} \\
                        \partial_{t} \bm{\eta}_{\bm{w}} - \bm{r}_{\bm{w}}
                    \end{pmatrix}
                }{
                    \begin{pmatrix}
                        \bm{\varphi}_{\tau, h}^{\bm{u}} \\
                        \bm{\varphi}_{\tau, h}^{\bm{w}}
                    \end{pmatrix}
                } \mathrm{d}t
                & \leq c \left( \tau^{k+1} + h^{\ell+1} \right) \norm{(\bm{u}, \bm{w}, p)}_{n}  \norm{\bm{M}_{\rho}^{1/2}
                    \begin{pmatrix}
                        \bm{\varphi}_{\tau, h}^{\bm{u}} \\
                        \bm{\varphi}_{\tau, h}^{\bm{w}}
                    \end{pmatrix}
                }_{L^{2}(I_{n}; \bm{L}^{2})}, \label{eq:dg_aux_estimates_rhs_I-1} \\
                \int_{I_{n}} a_{h}(\bm{\eta}_{\bm{u}}, \bm{\varphi}_{\tau, h}^{\bm{u}}) \mathrm{d}t
                & \leq c \tau^{k+2}  \norm{(\bm{u}, \bm{w}, p)}_{n}  \norm{\bm{\varphi}_{\tau, h}^{\bm{u}}}_{L^{2}(I_{n}; \bm{U}_{h})}, \label{eq:dg_aux_estimates_rhs_I-2} \\
                \int_{I_{n}} \scalar{\bm{K}^{-1} \bm{\eta}_{\bm{w}}}{\bm{\varphi}_{\tau, h}^{\bm{w}}} \mathrm{d}t
                & \leq c \left( \tau^{k+2} + h^{\ell+1} \right) \norm{(\bm{u}, \bm{w}, p)}_{n}  \norm{\bm{K}^{-1/2} \bm{\varphi}_{\tau, h}^{\bm{w}}}_{L^{2}(I_{n}; \bm{L}^{2})}, \label{eq:dg_aux_estimates_rhs_I-3} \\
                \int_{I_{n}} \scalar{\alpha  \eta_{p}}{\divv(\bm{\varphi}_{\tau, h}^{\bm{u}})} \mathrm{d}t
                & \leq c \tau^{k+2}  \norm{(\bm{u}, \bm{w}, p)}_{n}  \norm{\divv(\bm{\varphi}_{\tau, h}^{\bm{u}})}_{L^{2}(I_{n}; L^{2})}, \label{eq:dg_aux_estimates_rhs_I-4} \\
                \int_{I_{n}} \scalar{\eta_{p}}{\divv(\bm{\varphi}_{\tau, h}^{\bm{w}})} \mathrm{d}t
                & \leq c \tau^{k+2}  \norm{(\bm{u}, \bm{w}, p)}_{n}  \norm{\divv(\bm{\varphi}_{\tau, h}^{\bm{w}})}_{L^{2}(I_{n}; L^{2})}, \label{eq:dg_aux_estimates_rhs_I-5} \\
                \int_{I_{n}} \scalar{
                    \begin{pmatrix}
                        \bm{\mathcal{I}}_{\tau}^{*} \bm{f} - \bm{\mathcal{I}}_{\tau}^{\textnormal{GL}} \bm{f} \\
                        \bm{\mathcal{I}}_{\tau}^{*} \bm{g} - \bm{\mathcal{I}}_{\tau}^{\textnormal{GL}} \bm{g}
                    \end{pmatrix}
                }{
                    \begin{pmatrix}
                        \bm{\varphi}_{\tau, h}^{\bm{u}} \\
                        \bm{\varphi}_{\tau, h}^{\bm{w}}
                    \end{pmatrix}
                }
                \mathrm{d}t
                & \leq c \tau^{k+1} \norm{\partial_{t}^{k+1}
                    \begin{pmatrix}
                        \bm{f} \\
                        \bm{g}
                    \end{pmatrix}
                }_{L^{2}(I_{n}; \bm{L}^{2})}
                \norm{
                    \begin{pmatrix}
                        \bm{\varphi}_{\tau, h}^{\bm{u}} \\
                        \bm{\varphi}_{\tau, h}^{\bm{w}}
                    \end{pmatrix}
                }_{L^{2}(I_{n}; \bm{L}^{2})}, \label{eq:dg_aux_estimates_rhs_I-6}
            \end{align}
        \end{subequations}
        are satisfied for all $ \bm{\varphi}_{\tau, h}^{\bm{u}} \in \mathbb{P}_{k-1}(I_{n}; \bm{U}_{h}) $ and $ \bm{\varphi}_{\tau, h}^{\bm{w}} \in \mathbb{P}_{k-1}(I_{n}; \bm{W}_{h}) $.
    \end{lemma}
    
    \begin{proof}
        Using definition \eqref{eq:dg_splitting} and \eqref{eq:dg_definition_r_y} along with Lemma~\ref{lem:dg_identity_z}, the positive definiteness of $ \bm{M}_{\rho} $ and Lemma~\ref{lem:dg_projection_errors}, we obtain
        \begin{align*}
            & \quad \int_{I_{n}} \scalar{\bm{M}_{\rho}
                \begin{pmatrix}
                    \partial_{t} \bm{\eta}_{\bm{v}} - \bm{r}_{\bm{v}} \\
                    \partial_{t} \bm{\eta}_{\bm{w}} - \bm{r}_{\bm{w}}
                \end{pmatrix}
            }{
                \begin{pmatrix}
                    \bm{\varphi}_{\tau, h}^{\bm{u}} \\
                    \bm{\varphi}_{\tau, h}^{\bm{w}}
                \end{pmatrix}
            } \mathrm{d}t
            = \int_{I_{n}} \scalar{\bm{M}_{\rho}
                \begin{pmatrix}
                    \bm{\mathcal{I}}_{\tau}^{\text{GL}} \partial_{t} \bm{v} - \partial_{t} \bm{z}^{\bm{v}} \\
                    \bm{\mathcal{I}}_{\tau}^{\text{GL}} \partial_{t} \bm{w} - \partial_{t} \bm{z}^{\bm{w}}
                \end{pmatrix}
            }{
                \begin{pmatrix}
                    \bm{\varphi}_{\tau, h}^{\bm{u}} \\
                    \bm{\varphi}_{\tau, h}^{\bm{w}}
                \end{pmatrix}
            } \mathrm{d}t \\
            & = \int_{I_{n}} \scalar{\bm{M}_{\rho}
                \begin{pmatrix}
                    \bm{\mathcal{I}}_{\tau}^{\text{GL}} \partial_{t} \bm{v} - \bm{\mathcal{I}}_{\tau}^{\text{GL}} \bm{\mathcal{P}}_{1} \partial_{t} \bm{v} \\
                    \bm{\mathcal{I}}_{\tau}^{\text{GL}} \partial_{t} \bm{w} - \bm{\mathcal{I}}_{\tau}^{\text{GL}} \bm{\mathcal{P}}_{2} \partial_{t} \bm{w}
                \end{pmatrix}
            }{
                \begin{pmatrix}
                    \bm{\varphi}_{\tau, h}^{\bm{u}} \\
                    \bm{\varphi}_{\tau, h}^{\bm{w}}
                \end{pmatrix}
            } \mathrm{d}t \\
            & \leq c h^{\ell+1} \left( \norm{\bm{\mathcal{I}}_{\tau}^{\text{GL}} \partial_{t} \bm{v}}_{L^{2}(I_{n}; \bm{H}^{\ell+1})}
            + \norm{\bm{\mathcal{I}}_{\tau}^{\text{GL}} \partial_{t} \bm{w}}_{L^{2}(I_{n}; \bm{H}^{\ell+1})} \right) \norm{\bm{M}_{\rho}^{1/2}
                \begin{pmatrix}
                    \bm{\varphi}_{\tau, h}^{\bm{u}} \\
                    \bm{\varphi}_{\tau, h}^{\bm{w}}
                \end{pmatrix}
            }_{L^{2}(I_{n}; \bm{L}^{2})} \\
            & \leq c h^{\ell+1} \left( \tau^{k+1}  \norm{\partial_{t}^{k+3} \bm{u}}_{L^{2}(I_{n}; \bm{H}^{\ell+1})}
            + \norm{\partial_{t}^{2} \bm{u}}_{L^{2}(I_{n}; \bm{H}^{\ell+1})} \right. \\
            & \hspace{4.2em} \left. + \tau^{k+1}  \norm{\partial_{t}^{k+2} \bm{w}}_{L^{2}(I_{n}; \bm{H}^{\ell+1})}
            + \norm{\partial_{t} \bm{w}}_{L^{2}(I_{n}; \bm{H}^{\ell+1})} \right)
            \norm{\bm{M}_{\rho}^{1/2}
                \begin{pmatrix}
                    \bm{\varphi}_{\tau, h}^{\bm{u}} \\
                    \bm{\varphi}_{\tau, h}^{\bm{w}}
                \end{pmatrix}
            }_{L^{2}(I_{n}; \bm{L}^{2})},
        \end{align*}
        where we have used \eqref{eq:interpolation_error_GL-2} and $ \bm{v} = \partial_{t} \bm{u} $ in the last step.
        Exploiting the definition of $ \norm{\cdot}_{n} $ from \eqref{eq:dg_definition_norm_n} yields \eqref{eq:dg_aux_estimates_rhs_I-1}.
        Similar to \cite[Eqs.~(3.15)--(3.16)]{Karakashian2004Conv}, we conclude
        \begin{subequations}
            \label{eq:prf_dg_aux_estimates_rhs_I-1}
            \begin{align}
                \norm{\bm{\mathcal{I}}_{\tau}^{\text{GL}} \bm{y}}_{L^{2}(I_{n}; \bm{U}_{h})}
                & \leq c \norm{\bm{y}}_{L^{2}(I_{n}; \bm{U}_{h})} + c \tau \norm{\partial_{t} \bm{y}}_{L^{2}(I_{n}; \bm{U}_{h})}, \\
                \norm{\int_{t_{n-1}}^{t} \bm{y} \mathrm{d}s}_{L^{2}(I_{n}; \bm{U}_{h})}
                & \leq c \tau \norm{\bm{y}}_{L^{2}(I_{n}; \bm{U}_{h})},
            \end{align}
        \end{subequations}
        Next, we combine \eqref{eq:prf_dg_aux_estimates_rhs_I-1} with \eqref{eq:interpolation_error_GL-1} in order to get
        \begin{align}
            \label{eq:prf_dg_aux_estimates_rhs_I-2}
            \begin{aligned}
                \norm{\bm{\mathcal{I}}_{\tau}^{\text{GL}}\left( \int_{t_{n-1}}^{t} \bm{\mathcal{I}}_{\tau}^{\text{GL}} \partial_{t} \bm{u} - \partial_{t} \bm{u} \mathrm{d}s \right)}_{L^{2}(I_{n}; \bm{U}_{h})}
                & \leq c \tau \norm{\bm{\mathcal{I}}_{\tau}^{\text{GL}} \partial_{t} \bm{u} - \partial_{t} \bm{u}}_{L^{2}(I_{n}; \bm{U}_{h})}
                \leq c \tau^{k+2}  \norm{\partial_{t}^{k+2} \bm{u}}_{L^{2}(I_{n}; \bm{U}_{h})}
            \end{aligned}
        \end{align}
        as well as
        \begin{align*}
            \int_{I_{n}} a_{h}(\bm{\eta}_{\bm{u}}, \bm{\varphi}_{\tau, h}^{\bm{u}}) \mathrm{d}t
            & \leq c \norm{\bm{\mathcal{I}}_{\tau}^{\text{GL}} \bm{u} - \bm{\mathcal{I}}_{\tau}^{\text{GL}}\left( \int_{t_{n-1}}^{t} \bm{\mathcal{I}}_{\tau}^{\text{GL}} \partial_{t} \bm{u} \mathrm{d}s + \bm{u}(t_{n-1}) \right)}_{L^{2}(I_{n}; \bm{U}_{h})}  \norm{\bm{\varphi}_{\tau, h}^{\bm{u}}}_{L^{2}(I_{n}; \bm{U}_{h})} \\
            & \leq c \tau^{k+2}  \norm{(\bm{u}, \bm{w}, p)}_{n}  \norm{\bm{\varphi}_{\tau, h}^{\bm{u}}}_{L^{2}(I_{n}; \bm{U}_{h})},
        \end{align*}
        where we have used \eqref{eq:dg_projection_operators-1}.
        Lemma~\ref{lem:dg_estimates_eta} shows
        \begin{align*}
            \int_{I_{n}} \scalar{\bm{K}^{-1} \bm{\eta}_{\bm{w}}}{\bm{\varphi}_{\tau, h}^{\bm{w}}} \mathrm{d}t
            & = \int_{I_{n}} \scalar{\bm{K}^{-1/2} \bm{\eta}_{\bm{w}}}{\bm{K}^{-1/2} \bm{\varphi}_{\tau, h}^{\bm{w}}} \mathrm{d}t
            \leq c \left( \tau^{k+2} + h^{\ell+1} \right) \norm{(\bm{u}, \bm{w}, p)}_{n}  \norm{\bm{K}^{-1/2} \bm{\varphi}_{\tau, h}^{\bm{w}}}_{L^{2}(I_{n}; \bm{L}^{2})}.
        \end{align*}
        Combining \eqref{eq:dg_splitting} and \eqref{eq:dg_projection_operators-3} with arguments from \eqref{eq:prf_dg_estimates_eta-2} yields
        \begin{align*}
            & \quad \int_{I_{n}} \scalar{\alpha  \eta_{p}}{\divv(\bm{\varphi}_{\tau, h}^{\bm{u}})} \mathrm{d}t
            = \int_{I_{n}} \alpha \scalar{\mathcal{I}_{\tau}^{\text{GL}} \mathcal{P}_{3} p - z^{p}}{\divv(\bm{\varphi}_{\tau, h}^{\bm{u}})} \mathrm{d}t \\
            & \leq c \norm{\mathcal{I}_{\tau}^{\text{GL}} \mathcal{P}_{3} p - z^{p}}_{L^{2}(I_{n}; L^{2})}  \norm{\divv(\bm{\varphi}_{\tau, h}^{\bm{u}})}_{L^{2}(I_{n}; L^{2})}
            \leq c \tau^{k+2}  \norm{(\bm{u}, \bm{w}, p)}_{n}  \norm{\divv(\bm{\varphi}_{\tau, h}^{\bm{u}})}_{L^{2}(I_{n}; L^{2})},
        \end{align*}
        which shows \eqref{eq:dg_aux_estimates_rhs_I-4}.
        Assertion~\eqref{eq:dg_aux_estimates_rhs_I-5} can be proven analogously.
        Estimate~\eqref{eq:dg_aux_estimates_rhs_I-6} directly follows from \eqref{eq:interpolation_error_GL-1}.
    \end{proof}
    
    Next, we estimate the right-hand sides of \eqref{eq:dg_dyn_biot_var_eq-2} and \eqref{eq:dg_dyn_biot_var_eq-3}:
    
    \begin{lemma}[Estimates for \eqref{eq:dg_dyn_biot_var_eq-2}--\eqref{eq:dg_dyn_biot_var_eq-3}]
        \label{lem:dg_aux_estimates_rhs_II}
        The following estimates
        \begin{subequations}
            \label{eq:dg_aux_estimates_rhs_II}
            \begin{align}
                \int_{I_{n}} \scalar{\partial_{t} \bm{\eta}_{\bm{u}} - \bm{\eta}_{\bm{v}} - \bm{r}_{\bm{u}}}{\bm{\varphi}_{\tau, h}^{\bm{v}}} \mathrm{d}t
                & \leq c \tau^{k+2}  \norm{(\bm{u}, \bm{w}, p)}_{n}  \norm{\bm{\varphi}_{\tau, h}^{\bm{v}}}_{L^{2}(I_{n}; \bm{L}^{2})}, \label{eq:dg_aux_estimates_rhs_II-1} \\
                \int_{I_{n}} \alpha \scalar{\divv(\partial_{t} \bm{\eta}_{\bm{u}} - \bm{r}_{\bm{u}})}{\varphi_{\tau, h}^{p}} \mathrm{d}t
                & \leq c  h^{\ell+1}  \norm{(\bm{u}, \bm{w}, p)}_{n}  \norm{\varphi_{\tau, h}^{p}}_{L^{2}(I_{n}; L^{2})}, \label{eq:dg_aux_estimates_rhs_II-2} \\
                \int_{I_{n}} \scalar{\divv(\bm{\eta}_{\bm{w}})}{\varphi_{\tau, h}^{p}} \mathrm{d}t
                & \leq c \tau^{k+2}  \norm{(\bm{u}, \bm{w}, p)}_{n}  \norm{\varphi_{\tau, h}^{p}}_{L^{2}(I_{n}; L^{2})}, \label{eq:dg_aux_estimates_rhs_II-3}
            \end{align}
        \end{subequations}
        are valid for all $ \bm{\varphi}_{\tau, h}^{\bm{v}} \in \mathbb{P}_{k-1}(I_{n}; \bm{V}_{h}) $ and $ \varphi_{\tau, h}^{p} \in \mathbb{P}_{k-1}(I_{n}; P_{h}) $.
    \end{lemma}
    
    \begin{proof}
        Exploiting definitions \eqref{eq:dg_splitting} and \eqref{eq:dg_definition_r_y} coupled with $ \bm{v} = \partial_{t} \bm{u} $ and Lemma~\ref{lem:dg_identity_z} results in
        \begin{align*}
            \int_{I_{n}} \scalar{\partial_{t} \bm{\eta}_{\bm{u}} - \bm{\eta}_{\bm{v}} - \bm{r}_{\bm{u}}}{\bm{\varphi}_{\tau, h}^{\bm{v}}} \mathrm{d}t
            & = \int_{I_{n}} \scalar{\bm{z}^{\bm{v}} - \bm{\mathcal{I}}_{\tau}^{\text{GL}} \bm{\mathcal{P}}_{1} \bm{v}}{\bm{\varphi}_{\tau, h}^{\bm{v}}} \mathrm{d}t
            \leq c \tau^{k+2} \norm{(\bm{u}, \bm{w}, p)}_{n}  \norm{\bm{\varphi}_{\tau, h}^{\bm{v}}}_{L^{2}(I_{n}; \bm{L}^{2})},
        \end{align*}
        where we have used ideas from \eqref{eq:prf_dg_estimates_eta-2} in the last step which shows \eqref{eq:dg_aux_estimates_rhs_II-1}.
        Using Lemma~\ref{lem:dg_identity_z}, Lemma~\ref{lem:dg_projection_errors} and \eqref{eq:interpolation_error_GL-2} yields
        \begin{align*}
            \int_{I_{n}} \alpha \scalar{\divv(\partial_{t} \bm{\eta}_{\bm{u}} - \bm{r}_{\bm{u}})}{\varphi_{\tau, h}^{p}} \mathrm{d}t
            & = \int_{I_{n}} \alpha \scalar{\divv(\bm{\mathcal{I}}_{\tau}^{\text{GL}} \partial_{t} \bm{u} - \bm{\mathcal{I}}_{\tau}^{\text{GL}} \bm{\mathcal{P}}_{1} \partial_{t} \bm{u})}{\varphi_{\tau, h}^{p}} \mathrm{d}t \\
            & \leq c h^{\ell+1}  \norm{\bm{\mathcal{I}}_{\tau}^{\text{GL}} \partial_{t} \bm{u}}_{L^{2}(I_{n}; \bm{H}^{\ell+2})}  \norm{\varphi_{\tau, h}^{p}}_{L^{2}(I_{n}; L^{2})} \\
            & \leq c h^{\ell+1} \left( \tau^{k+1}  \norm{\partial_{t}^{k+2} \bm{u}}_{L^{2}(I_{n}; \bm{H}^{\ell+2})}
            + \norm{\partial_{t} \bm{u}}_{L^{2}(I_{n}; \bm{H}^{\ell+2})} \right)  \norm{\varphi_{\tau, h}^{p}}_{L^{2}(I_{n}; L^{2})},
        \end{align*}
        from which \eqref{eq:dg_aux_estimates_rhs_II-2} follows.
        Combining \eqref{eq:dg_splitting} with \eqref{eq:dg_projection_operators-2} and following the lines of \eqref{eq:prf_dg_estimates_eta-2} leads to
        \begin{align*}
            \int_{I_{n}} \scalar{\divv(\bm{\eta}_{\bm{w}})}{\varphi_{\tau, h}^{p}} \mathrm{d}t
            & = \int_{I_{n}} \scalar{\divv(\bm{\mathcal{I}}_{\tau}^{\text{GL}} \bm{\mathcal{P}}_{2} \bm{w} - \bm{z}^{\bm{w}})}{\varphi_{\tau, h}^{p}} \mathrm{d}t \\
            & \leq \norm{\divv(\bm{\mathcal{I}}_{\tau}^{\text{GL}} \bm{\mathcal{P}}_{2} \bm{w} - \bm{z}^{\bm{w}})}_{L^{2}(I_{n}; L^{2})}  \norm{\varphi_{\tau, h}^{p}}_{L^{2}(I_{n}; L^{2})} \\
            & \leq c \tau^{k+2} \left( h^{\ell+1}  \norm{\partial_{t}^{k+2} \bm{w}}_{L^{2}(I_{n}; \bm{H}^{\ell+2})}
            + \norm{\partial_{t}^{k+2} \bm{w}}_{L^{2}(I_{n}; \bm{H}^{1})} \right) \norm{\varphi_{\tau, h}^{p}}_{L^{2}(I_{n}; L^{2})},
        \end{align*}
        which proves \eqref{eq:dg_aux_estimates_rhs_II-3}.
    \end{proof}
    
    Below, we estimate several one-sided limits arising during our error analysis.
    
    \begin{lemma}
        \label{lem:dg_estimate_limits}
        It holds
        \begin{subequations}
            \label{eq:dg_estimate_limits}
            \begin{align}
                \norm{\bm{e}_{\bm{u}}(t_{n-1}^{+})}_{\bm{U}_{h}}^{2}
                & \leq (1+\tau\varepsilon_{1}) \norm{\bm{e}_{\bm{u}}(t_{n-1})}_{\bm{U}_{h}}^{2}
                + c \tau^{2(k+1)} \norm{(\bm{u}, \bm{w}, p)}_{n}^{2}, \label{eq:dg_estimate_limits-1} \\
                \norm{\bm{e}_{\bm{v}}(t_{n-1}^{+})}_{\bm{L}^{2}(\Omega)}^{2}
                & \leq (1+\tau\varepsilon_{1}) \norm{\bm{e}_{\bm{v}}(t_{n-1})}_{\bm{L}^{2}(\Omega)}^{2}
                + c  \tau^{2(k+1)}  \norm{(\bm{u}, \bm{w}, p)}_{n}^{2}, \label{eq:dg_estimate_limits-2} \\
                \norm{\bm{e}_{\bm{w}}(t_{n-1}^{+})}_{\bm{L}^{2}(\Omega)}^{2}
                & \leq (1+\tau\varepsilon_{1}) \norm{\bm{e}_{\bm{w}}(t_{n-1})}_{\bm{L}^{2}(\Omega)}^{2}
                + c  \tau^{2(k+1)}  \norm{(\bm{u}, \bm{w}, p)}_{n}^{2}, \label{eq:dg_estimate_limits-3} \\
                \norm{e_{p}(t_{n-1}^{+})}_{L^{2}(\Omega)}^{2}
                & \leq (1+\tau\varepsilon_{1}) \norm{e_{p}(t_{n-1})}_{L^{2}(\Omega)}^{2}
                + c  \tau^{2(k+1)}  \norm{(\bm{u}, \bm{w}, p)}_{n}^{2}, \label{eq:dg_estimate_limits-4} \\
                \norm{\bm{M}_{\rho}^{1/2}
                    \begin{pmatrix}
                        \bm{e}_{\bm{v}}(t_{n-1}^{+}) \\
                        \bm{e}_{\bm{w}}(t_{n-1}^{+})
                    \end{pmatrix}
                }_{\bm{L}^{2}(\Omega)}^{2}
                & \leq (1 + \tau \varepsilon_{1}) \norm{\bm{M}_{\rho}^{1/2}
                    \begin{pmatrix}
                        \bm{e}_{\bm{v}}(t_{n-1}) \\
                        \bm{e}_{\bm{w}}(t_{n-1})
                    \end{pmatrix}
                }_{\bm{L}^{2}(\Omega)}^{2}
                + c  \tau^{2(k+1)}  \norm{(\bm{u}, \bm{w}, p)}_{n}^{2}, \label{eq:dg_estimate_limits-5} \\
                a_{h}(\bm{e}_{\bm{u}}(t_{n-1}^{+}), \bm{e}_{\bm{u}}(t_{n-1}^{+}))
                & \leq a_{h}(\bm{e}_{\bm{u}}(t_{n-1}), \bm{e}_{\bm{u}}(t_{n-1}))
                + \tau \varepsilon_{1} \norm{\bm{e}_{\bm{u}}(t_{n-1})}_{\bm{U}_{h}}^{2}
                + c  \tau^{2(k+1)}   \norm{(\bm{u}, \bm{w}, p)}_{n}^{2}, \label{eq:dg_estimate_limits-6}
            \end{align}
        \end{subequations}
        for $ n = 2, \ldots, N $ and $ \varepsilon_{1} > 0 $.
    \end{lemma}
    
    \begin{proof}
        Using \eqref{eq:dg_splitting} along with the triangle inequality and continuity of the discrete solution, we obtain
        \begin{align*}
            \norm{\bm{e}_{\bm{u}}(t_{n-1}^{+})}_{\bm{U}_{h}}
            & \leq \norm{\bm{u}_{\tau, h}(t_{n-1}) - \bm{z}^{\bm{u}}(t_{n-1})}_{\bm{U}_{h}}
            + \norm{\bm{z}^{\bm{u}}(t_{n-1}) - \bm{z}^{\bm{u}}(t_{n-1}^{+})}_{\bm{U}_{h}}.
        \end{align*}
        Now we apply \eqref{eq:interpolation_error_GL-1} and Lemma~\ref{lem:dg_projection_errors} to get
        \begin{align}
            \label{eq:prf_dg_estimate_limits-1}
            \begin{aligned}
                \norm{\bm{z}^{\bm{u}}(t_{n-1}) - \bm{z}^{\bm{u}}(t_{n-1}^{+})}_{\bm{U}_{h}}
                & = \norm{\int_{t_{n-2}}^{t_{n-1}} \bm{\mathcal{I}}_{\tau}^{\text{GL}} \bm{\mathcal{P}}_{1} \partial_{t} \bm{u} - \bm{\mathcal{P}}_{1} \partial_{t} \bm{u} \mathrm{d}s}_{\bm{U}_{h}} \\
                & \leq c \tau^{1/2} \tau^{k+1} \left( h^{\ell+1}  \norm{\partial_{t}^{k+2} \bm{u}}_{L^{2}(I_{n}; \bm{H}^{\ell+2})}
                + \norm{\partial_{t}^{k+2} \bm{u}}_{L^{2}(I_{n}; \bm{U}_{h})} \right)
            \end{aligned}
        \end{align}
        and consequently after squaring we obtain \eqref{eq:dg_estimate_limits-1}.
        The estimates \eqref{eq:dg_estimate_limits-2}--\eqref{eq:dg_estimate_limits-4} can be shown analogously.
        
        Using the symmetry of $ \bm{M}_{\rho} $, we get
        \begin{align*}
            \norm{\bm{M}_{\rho}^{1/2}
                \begin{pmatrix}
                    \bm{e}_{\bm{v}}(t_{n-1}^{+}) \\
                    \bm{e}_{\bm{w}}(t_{n-1}^{+})
                \end{pmatrix}
            }_{\bm{L}^{2}(\Omega)}^{2}
            & = \norm{\bm{M}_{\rho}^{1/2}
                \begin{pmatrix}
                    \bm{e}_{\bm{v}}(t_{n-1}) \\
                    \bm{e}_{\bm{w}}(t_{n-1})
                \end{pmatrix}
            }_{\bm{L}^{2}(\Omega)}^{2}
            + \norm{\bm{M}_{\rho}^{1/2}
                \begin{pmatrix}
                    \bm{e}_{\bm{v}}(t_{n-1}^{+}) - \bm{e}_{\bm{v}}(t_{n-1}) \\
                    \bm{e}_{\bm{w}}(t_{n-1}^{+}) - \bm{e}_{\bm{w}}(t_{n-1})
                \end{pmatrix}
            }_{\bm{L}^{2}(\Omega)}^{2} \\
            & \quad + 2 \scalar{\bm{M}_{\rho}
                \begin{pmatrix}
                    \bm{e}_{\bm{v}}(t_{n-1}^{+}) - \bm{e}_{\bm{v}}(t_{n-1}) \\
                    \bm{e}_{\bm{w}}(t_{n-1}^{+}) - \bm{e}_{\bm{w}}(t_{n-1})
                \end{pmatrix}
            }{
                \begin{pmatrix}
                    \bm{e}_{\bm{v}}(t_{n-1}) \\
                    \bm{e}_{\bm{w}}(t_{n-1})
                \end{pmatrix}
            } \\
            & \leq (1 + \tau \varepsilon_{1}) \norm{\bm{M}_{\rho}^{1/2}
                \begin{pmatrix}
                    \bm{e}_{\bm{v}}(t_{n-1}) \\
                    \bm{e}_{\bm{w}}(t_{n-1})
                \end{pmatrix}
            }_{\bm{L}^{2}(\Omega)}^{2} \\
            & \quad + (1 + (\tau \varepsilon_{1})^{-1}) \norm{\bm{M}_{\rho}^{1/2}
                \begin{pmatrix}
                    \bm{e}_{\bm{v}}(t_{n-1}^{+}) - \bm{e}_{\bm{v}}(t_{n-1}) \\
                    \bm{e}_{\bm{w}}(t_{n-1}^{+}) - \bm{e}_{\bm{w}}(t_{n-1})
                \end{pmatrix}
            }_{\bm{L}^{2}(\Omega)}^{2}.
        \end{align*}
        Inserting \eqref{eq:dg_splitting} together with the continuity of the discrete solution, we obtain
        \begin{align*}
            \bm{e}_{\bm{v}}(t_{n-1}^{+}) - \bm{e}_{\bm{v}}(t_{n-1})
            = \bm{z}^{\bm{v}}(t_{n-1}) - \bm{z}^{\bm{v}}(t_{n-1}^{+}).
        \end{align*}
        Then, we have with \eqref{eq:prf_dg_estimate_limits-1}
        \begin{align*}
            \norm{\bm{M}_{\rho}^{1/2}
                \begin{pmatrix}
                    \bm{e}_{\bm{v}}(t_{n-1}^{+}) - \bm{e}_{\bm{v}}(t_{n-1}) \\
                    \bm{e}_{\bm{w}}(t_{n-1}^{+}) - \bm{e}_{\bm{w}}(t_{n-1})
                \end{pmatrix}
            }_{\bm{L}^{2}(\Omega)}^{2}
            & \leq c \norm{\bm{z}^{\bm{v}}(t_{n-1}^{+}) - \bm{z}^{\bm{v}}(t_{n-1})}_{\bm{L}^{2}(\Omega)}^{2}
            + c \norm{\bm{z}^{\bm{w}}(t_{n-1}^{+}) - \bm{z}^{\bm{w}}(t_{n-1})}_{\bm{L}^{2}(\Omega)}^{2} \\
            & \leq c \tau \tau^{2(k+1)} \norm{(\bm{u}, \bm{w}, p)}_{n}^{2},
        \end{align*}
        which proves \eqref{eq:dg_estimate_limits-5}.
        Similarly, it follows \eqref{eq:dg_estimate_limits-6}.
    \end{proof}
    
    Subsequently, we use a specific choice of test functions for \eqref{eq:dg_dyn_biot_var_eq} and exploit the auxiliary results presented in Section~\ref{subsec:auxiliary_results}.
    
    \begin{lemma}
        \label{lem:dg_estimate_tf1}
        The following estimate
        \begin{align}
            \label{eq:dg_estimate_tf1}
            \begin{aligned}
                & \quad  a_{h}(\bm{e}_{\bm{u}}(t_{n}), \bm{e}_{\bm{u}}(t_{n}))
                +  \norm{\bm{M}_{\rho}^{1/2}
                    \begin{pmatrix}
                        \bm{e}_{\bm{v}}(t_{n}) \\
                        \bm{e}_{\bm{w}}(t_{n})
                    \end{pmatrix}
                }_{\bm{L}^{2}(\Omega)}^{2}
                + s_{0} \norm{e_{p}(t_{n})}_{L^{2}(\Omega)}^{2}
                + \norm{\bm{K}^{-1/2} \bm{e}_{\bm{w}}}_{L^{2}(I_{n}; \bm{L}^{2})}^{2} \\
                & \leq  a_{h}(\bm{e}_{\bm{u}}(t_{n-1}), \bm{e}_{\bm{u}}(t_{n-1}))
                +  \norm{\bm{M}_{\rho}^{1/2}
                    \begin{pmatrix}
                        \bm{e}_{\bm{v}}(t_{n-1}) \\
                        \bm{e}_{\bm{w}}(t_{n-1})
                    \end{pmatrix}
                }_{\bm{L}^{2}(\Omega)}^{2}
                + s_{0} \norm{e_{p}(t_{n-1})}_{L^{2}(\Omega)}^{2} \\
                & \quad +  \tau \varepsilon_{1} \left( \norm{\bm{e}_{\bm{u}}(t_{n-1})}_{\bm{U}_{h}}^{2}
                + \norm{\bm{M}_{\rho}^{1/2}
                    \begin{pmatrix}
                        \bm{e}_{\bm{v}}(t_{n-1}) \\
                        \bm{e}_{\bm{w}}(t_{n-1})
                    \end{pmatrix}
                }_{\bm{L}^{2}(\Omega)}^{2}
                + s_{0} \norm{e_{p}(t_{n-1})}_{L^{2}(\Omega)}^{2} \right) \\
                & \quad + \varepsilon_{2} \left( \norm{\bm{e}_{\bm{u}}}_{L^{2}(I_{n}; \bm{U}_{h})}^{2}
                + \norm{\bm{M}_{\rho}^{1/2}
                    \begin{pmatrix}
                        \bm{e}_{\bm{v}} \\
                        \bm{e}_{\bm{w}}
                \end{pmatrix}}_{L^{2}(I_{n}; \bm{L}^{2})}^{2}
                + s_{0} \norm{e_{p}}_{L^{2}(I_{n}; L^{2})}^{2}
                + \tau^{2}  c_{2} \norm{\divv(\bm{\mathcal{I}}_{\tau}^{\textnormal{G}} \bm{e}_{\bm{w}})}_{L^{2}(I_{n}; L^{2})}^{2} \right) \\
                & \quad + \varepsilon_{2} \norm{\bm{K}^{-1/2} \bm{e}_{\bm{w}}}_{L^{2}(I_{n}; \bm{L}^{2})}^{2}
                + c \left( \tau^{2(k+1)} + h^{2(\ell+1)} \right) \norm{(\bm{u}, \bm{w}, p)}_{n}^{2}
                + c \tau^{2(k+1)} \norm{\partial_{t}^{k+1} 
                    \begin{pmatrix}
                        \bm{f} \\
                        \bm{g}
                    \end{pmatrix}
                }_{L^{2}(I_{n}; \bm{L}^{2})}^{2}
            \end{aligned}
        \end{align}
        holds for $ n = 1, \ldots, N $ and $ \varepsilon_{1}, \varepsilon_{2}, c_{2}, c > 0 $.
    \end{lemma}
    
    \begin{proof}
        Using $ \bm{\varphi}_{\tau, h}^{\bm{v}} := \partial_{t} \bm{e}_{\bm{u}} \in \mathbb{P}_{k-1}(I_{n}; \bm{V}_{h}) $ in \eqref{eq:dg_dyn_biot_var_eq-2} and Lemma~\ref{lem:dg_aux_estimates_rhs_II} results in
        \begin{align*}
            \int_{I_{n}} \scalar{\partial_{t} \bm{e}_{\bm{u}}}{\partial_{t} \bm{e}_{\bm{u}}} \mathrm{d}t
            & = \int_{I_{n}} \scalar{\bm{e}_{\bm{v}}}{\partial_{t} \bm{e}_{\bm{u}}}
            + \scalar{\partial_{t} \bm{\eta}_{\bm{u}} - \bm{\eta}_{\bm{v}} - \bm{r}_{\bm{u}}}{\partial_{t} \bm{e}_{\bm{u}}} \mathrm{d}t \\
            & \leq (\norm{\bm{e}_{\bm{v}}}_{L^{2}(I_{n}; \bm{L}^{2})} + c \tau^{k+2}  \norm{(\bm{u}, \bm{w}, p)}_{n} ) \norm{\partial_{t} \bm{e}_{\bm{u}}}_{L^{2}(I_{n}; \bm{L}^{2})}.
        \end{align*}
        So we get
        \begin{align}
            \label{eq:prf_dg_estimate_tf1-1}
            \norm{\partial_{t} \bm{e}_{\bm{u}}}_{L^{2}(I_{n}; \bm{L}^{2})}
            & \leq \norm{\bm{e}_{\bm{v}}}_{L^{2}(I_{n}; \bm{L}^{2})}
            + c \tau^{k+2}  \norm{(\bm{u}, \bm{w}, p)}_{n}.
        \end{align}
        We choose the test functions as
        \begin{align}
            \label{eq:prf_dg_estimate_tf1-2}
            \bm{\varphi}_{\tau, h}^{\bm{u}} & := \partial_{t} \bm{e}_{\bm{u}}, &
            \bm{\varphi}_{\tau, h}^{\bm{v}} & := - (\bar{\rho}  \partial_{t} \bm{e}_{\bm{v}} + \rho_{f}  \partial_{t} \bm{e}_{\bm{w}}), &
            \bm{\varphi}_{\tau, h}^{\bm{w}} & := \bm{\mathcal{I}}_{\tau}^{\text{G}} \bm{e}_{\bm{w}}, &
            \varphi_{\tau, h}^{p} & := \mathcal{I}_{\tau}^{\text{G}} e_{p}.
        \end{align}
        Using similar arguments as in the proof of \eqref{eq:dg_auxiliary_III-5}, it can be shown that
        \begin{subequations}
            \label{eq:prf_dg_estimate_tf1-3}
            \begin{align}
                \norm{\bm{\varphi}_{\tau, h}^{\bm{u}}}_{L^{2}(I_{n}; \bm{U}_{h})}
                & \leq c \tau^{-1} \norm{\bm{e}_{\bm{u}}}_{L^{2}(I_{n}; \bm{U}_{h})}, \\
                \norm{\bm{\varphi}_{\tau, h}^{\bm{v}}}_{L^{2}(I_{n}; \bm{L}^{2})}
                & \leq c \tau^{-1} \left( \norm{\bm{e}_{\bm{v}}}_{L^{2}(I_{n}; \bm{L}^{2})}
                + \norm{\bm{e}_{\bm{w}}}_{L^{2}(I_{n}; \bm{L}^{2})} \right).
            \end{align}
        \end{subequations}
        Application of \eqref{eq:dg_auxiliary_I-3} leads to
        \begin{align}
            \label{eq:prf_dg_estimate_tf1-4}
            \norm{\bm{\varphi}_{\tau, h}^{\bm{w}}}_{L^{2}(I_{n}; \bm{L}^{2})}
            \leq c \norm{\bm{e}_{\bm{w}}}_{L^{2}(I_{n}; \bm{L}^{2})},
            \qquad \text{and} \qquad
            \norm{\varphi_{\tau, h}^{p}}_{L^{2}(I_{n}; L^{2})}
            \leq c \norm{e_{p}}_{L^{2}(I_{n}; L^{2})}.
        \end{align}
        Exploiting \eqref{eq:dg_auxiliary_I-1} along with the definition of $ \bm{M}_{\rho} $, we have
        \begin{align}
            \label{eq:prf_dg_estimate_tf1-5}
            \begin{aligned}
                & \quad \int_{I_{n}} \scalar{\bm{M}_{\rho}
                    \begin{pmatrix}
                        \partial_{t} \bm{e}_{\bm{v}} \\
                        \partial_{t} \bm{e}_{\bm{w}}
                    \end{pmatrix}
                }{
                    \begin{pmatrix}
                        \partial_{t} \bm{e}_{\bm{u}} \\
                        \bm{\mathcal{I}}_{\tau}^{\text{G}} \bm{e}_{\bm{w}}
                    \end{pmatrix}
                }
                - \scalar{\partial_{t} \bm{e}_{\bm{u}} - \bm{e}_{\bm{v}}}{\bar{\rho}  \partial_{t} \bm{e}_{\bm{v}} + \rho_{f}  \partial_{t} \bm{e}_{\bm{w}}} \mathrm{d}t \\
                & = \int_{I_{n}} \scalar{\bm{M}_{\rho}
                    \begin{pmatrix}
                        \partial_{t} \bm{e}_{\bm{v}} \\
                        \partial_{t} \bm{e}_{\bm{w}}
                    \end{pmatrix}
                }{
                    \begin{pmatrix}
                        \bm{e}_{\bm{v}} \\
                        \bm{e}_{\bm{w}}
                    \end{pmatrix}
                } \mathrm{d}t
                = \frac{1}{2} \norm{\bm{M}_{\rho}^{1/2}
                    \begin{pmatrix}
                        \bm{e}_{\bm{v}}(t_{n}) \\
                        \bm{e}_{\bm{w}}(t_{n})
                    \end{pmatrix}
                }_{\bm{L}^{2}(\Omega)}^{2}
                - \frac{1}{2} \norm{\bm{M}_{\rho}^{1/2}
                    \begin{pmatrix}
                        \bm{e}_{\bm{v}}(t_{n-1}^{+}) \\
                        \bm{e}_{\bm{w}}(t_{n-1}^{+})
                    \end{pmatrix}
                }_{\bm{L}^{2}(\Omega)}^{2},
            \end{aligned}
        \end{align}
        and
        \begin{align}
            \label{eq:prf_dg_estimate_tf1-6}
            \int_{I_{n}} \scalar{s_{0}  \partial_{t} e_{p}}{\varphi_{\tau, h}^{p}} \mathrm{d}t
            & = \int_{I_{n}} \scalar{s_{0}  \partial_{t} e_{p}}{e_{p}} \mathrm{d}t
            = \frac{s_{0}}{2} \norm{e_{p}(t_{n})}_{L^{2}(\Omega)}^{2}
            - \frac{s_{0}}{2} \norm{e_{p}(t_{n-1}^{+})}_{L^{2}(\Omega)}^{2}.
        \end{align}
        Similarly, we get with the symmetry of $ a_{h}(\cdot, \cdot) $
        \begin{align}
            \label{eq:prf_dg_estimate_tf1-7}
            \int_{I_{n}} a_{h}(\bm{e}_{\bm{u}}, \bm{\varphi}_{\tau, h}^{\bm{u}}) \mathrm{d}t
            & = \int_{I_{n}} a_{h}(\bm{e}_{\bm{u}}, \partial_{t} \bm{e}_{\bm{u}}) \mathrm{d}t
            = \frac{1}{2}  a_{h}(\bm{e}_{\bm{u}}(t_{n}), \bm{e}_{\bm{u}}(t_{n}))
            - \frac{1}{2}  a_{h}(\bm{e}_{\bm{u}}(t_{n-1}^{+}), \bm{e}_{\bm{u}}(t_{n-1}^{+})).
        \end{align}
        From \eqref{eq:dg_auxiliary_I-1}, we get
        \begin{align}
            \label{eq:prf_dg_estimate_tf1-8}
            \begin{aligned}
                \int_{I_{n}} \scalar{\bm{K}^{-1} \bm{e}_{\bm{w}}}{\bm{\varphi}_{\tau, h}^{\bm{w}}} \mathrm{d}t
                & = \norm{\bm{K}^{-1/2} \bm{e}_{\bm{w}}}_{L^{2}(I_{n}; \bm{L}^{2})}^{2}.
            \end{aligned}
        \end{align}
        Exploiting \eqref{eq:dg_auxiliary_I-1}, we conclude
        \begin{align}
            \label{eq:prf_dg_estimate_tf1-9}
            \begin{aligned}
                \int_{I_{n}} \scalar{\alpha \divv(\partial_{t} \bm{e}_{\bm{u}})}{\varphi_{\tau, h}^{p}} \mathrm{d}t
                & = \int_{I_{n}} \scalar{\alpha \divv(\partial_{t} \bm{e}_{\bm{u}})}{\mathcal{I}_{\tau}^{\text{G}} e_{p}} \mathrm{d}t
                = \int_{I_{n}} \scalar{\alpha \divv(\partial_{t} \bm{e}_{\bm{u}})}{e_{p}} \mathrm{d}t \\
                & = \int_{I_{n}} \scalar{\alpha  e_{p}}{\divv(\bm{\varphi}_{\tau, h}^{\bm{u}})} \mathrm{d}t.
            \end{aligned}
        \end{align}
        Using \eqref{eq:dg_auxiliary_I-2}, we get
        \begin{align}
            \label{eq:prf_dg_estimate_tf1-10}
            \int_{I_{n}} \scalar{\divv(\bm{e}_{\bm{w}})}{\varphi_{\tau, h}^{p}} \mathrm{d}t
            & = \int_{I_{n}} \scalar{\divv(\bm{e}_{\bm{w}})}{\mathcal{I}_{\tau}^{\text{G}} e_{p}} \mathrm{d}t
            = \int_{I_{n}} \scalar{e_{p}}{\divv(\bm{\mathcal{I}}_{\tau}^{\text{G}} \bm{e}_{\bm{w}})} \mathrm{d}t
            = \int_{I_{n}} \scalar{e_{p}}{\divv(\bm{\varphi}_{\tau, h}^{\bm{w}})} \mathrm{d}t.
        \end{align}
        From equations \eqref{eq:prf_dg_estimate_tf1-1}--\eqref{eq:prf_dg_estimate_tf1-10}, \eqref{eq:dg_vanishing_term}, Lemma~\ref{lem:dg_aux_estimates_rhs_I}--Lemma~\ref{lem:dg_estimate_limits} and Young's inequality with $ \varepsilon_{2} > 0 $, we infer \eqref{eq:dg_estimate_tf1}.
    \end{proof}
    
    With a second choice of test functions, we can prove the following estimate:
    
    \begin{lemma}
        \label{lem:dg_estimate_tf2}
        For $ n = 1, \ldots, N $ there holds
        \begin{align}
            \label{eq:dg_estimate_tf2}
            \begin{aligned}
                & \quad C_{L^{2}} \left( \norm{\bm{e}_{\bm{u}}}_{L^{2}(I_{n}; \bm{U}_{h})}^{2}
                + \norm{\bm{M}_{\rho}^{1/2}
                    \begin{pmatrix}
                        \bm{e}_{\bm{v}} \\
                        \bm{e}_{\bm{w}}
                    \end{pmatrix}
                }_{L^{2}(I_{n}; \bm{L}^{2})}^{2}
                + s_{0} \norm{e_{p}}_{L^{2}(I_{n}; L^{2})}^{2} \right)
                + \frac{1}{2} \tau^{2} c_{2} \norm{\divv(\bm{\mathcal{I}}_{\tau}^{\textnormal{G}} \bm{e}_{\bm{w}})}_{L^{2}(I_{n}; L^{2})}^{2} \\
                & \quad + \tau (C_{L^{2}} - c_{1} C_{\infty} C_{2}(1 + \varepsilon_{1} \tau)) \norm{\bm{K}^{-1/2} \bm{e}_{\bm{w}}}_{L^{2}(I_{n}; \bm{L}^{2})}^{2} \\
                & \leq C_{2} c_{1} \tau (1 + \tau \varepsilon_{1}) \left( \norm{\bm{e}_{\bm{u}}(t_{n-1})}_{\bm{U}_{h}}^{2}
                + \norm{\bm{M}_{\rho}^{1/2}
                    \begin{pmatrix}
                        \bm{e}_{\bm{v}}(t_{n-1}) \\
                        \bm{e}_{\bm{w}}(t_{n-1})
                    \end{pmatrix}
                }_{\bm{L}^{2}(\Omega)}^{2}
                + s_{0} \norm{e_{p}(t_{n-1})}_{L^{2}(\Omega)}^{2} \right) \\
                & \quad + \tau \varepsilon_{2} \left(
                \norm{\bm{e}_{\bm{u}}}_{L^{2}(I_{n}; \bm{U}_{h})}^{2}
                + \norm{\bm{M}_{\rho}^{1/2}
                    \begin{pmatrix}
                        \bm{e}_{\bm{v}} \\
                        \bm{e}_{\bm{w}}
                \end{pmatrix}}_{L^{2}(I_{n}; \bm{L}^{2})}^{2}
                + s_{0} \norm{e_{p}}_{L^{2}(I_{n}; L^{2})}^{2} \right) \\
                & \quad + \frac{1}{2} \varepsilon_{2} \tau^{2} c_{2} \norm{\divv(\bm{\mathcal{I}}_{\tau}^{\textnormal{G}} \bm{e}_{\bm{w}})}_{L^{2}(I_{n}; L^{2})}^{2}
                + \tau \varepsilon_{2} \norm{\bm{K}^{-1/2} \bm{e}_{\bm{w}}}_{L^{2}(I_{n}; \bm{L}^{2})}^{2}
                + c \tau \left( \tau^{2(k+1)} + h^{2(\ell+1)} \right) \norm{(\bm{u}, \bm{w}, p)}_{n}^{2} \\
                & \quad + \tau^{2k+3} \norm{\partial_{t}^{k+1}
                    \begin{pmatrix}
                        \bm{f} \\
                        \bm{g}
                    \end{pmatrix}
                }_{L^{2}(I_{n}; \bm{L}^{2})}^{2}
            \end{aligned}
        \end{align}
        with $ C_{L^{2}} := \frac{1}{2} c_{1} C_{1} $.
    \end{lemma}
    
    \begin{proof}
        Here, we use the test functions
        \begin{align}
            \label{eq:prf_dg_estimate_tf2-1}
            \begin{aligned}
                \bm{\varphi}_{\tau, h}^{\bm{u}} & := c_{1} \tau  \bm{e}_{\bm{u}, \beta}, & \quad
                \bm{\varphi}_{\tau, h}^{\bm{v}} & := - c_{1} \tau  (\bar{\rho}  \bm{e}_{\bm{v}, \beta} + \rho_{f}  \bm{e}_{\bm{w}, \beta}), & \quad
                \bm{\varphi}_{\tau, h}^{\bm{w}} & := c_{1} \tau \sum_{i=1}^{k} \beta_{i}  \bm{e}_{\bm{w}}(t_{n, i}^{\text{G}})  L_{n, i}^{\text{G}}, \\
                \varphi_{\tau, h}^{p} & := \varphi^{p}_{1} + \varphi^{p}_{2}, & \quad
                \varphi^{p}_{1} & := c_{1} \tau \sum_{i=1}^{k} \beta_{i}  e_{p}(t_{n, i}^{\text{G}})  L_{n, i}^{\text{G}}, & \quad
                \varphi^{p}_{2} & := c_{2} \tau^{2} \divv(\bm{\mathcal{I}}_{\tau}^{\text{G}} \bm{e}_{\bm{w}}),
            \end{aligned}
        \end{align}
        with
        \begin{align*}
            c_{1} & < \frac{1}{C_{2} (1 + \tau \varepsilon_{1})} \min\{C_{L^{2}} - \varepsilon_{1}, \frac{1}{4 C_{\infty}}\}, &
            c_{2} & \leq \frac{C_{1}}{2 C_{3}} \min\{\frac{\lambda}{\alpha^{2}}, \frac{1}{s_{0}^{2}}\},
        \end{align*}
        where $ C_{\infty} $ denotes the constant in the $ L^{2} $--$ L^{\infty} $ inequality.
        Further, we set
        \begin{align*}
            e_{y, \beta} & := \sum_{i=0}^{k} \sum_{j=1}^{k} \beta_{j}  e_{y}(t_{n, i}^{\text{G}, 0})  \partial_{t} L_{n, i}^{\text{G}, 0}(t_{n, j}^{\text{G}})  L_{n, j}^{\text{G}}(t),
        \end{align*}
        as well as $ \beta_{j} := (\hat{t}_{j}^{\text{G}})^{-1} $ for $ j = 1, \ldots, k $ and $ \beta_{0} := 1 $.
        Exploiting \eqref{eq:dg_auxiliary_III-4} and \eqref{eq:prf_dg_estimate_tf1-1} shows
        \begin{align}
            \label{eq:prf_dg_estimate_tf2-2}
            \norm{\bm{\varphi}_{\tau, h}^{\bm{u}}}_{L^{2}(I_{n}; \bm{L}^{2})}
            & \leq c c_{1} \tau \norm{\partial_{t} \bm{e}_{\bm{u}}}_{L^{2}(I_{n}; \bm{L}^{2})}
            \leq c c_{1} \tau \norm{\bm{e}_{\bm{v}}}_{L^{2}(I_{n}; \bm{L}^{2})}
            + c \tau^{k+2}  \norm{(\bm{u}, \bm{w}, p)}_{n}.
        \end{align}
        Similar to the proofs of \eqref{eq:dg_auxiliary_III-4}--\eqref{eq:dg_auxiliary_III-5}, it follows
        \begin{align}
            \label{eq:prf_dg_estimate_tf2-3}
            \norm{\bm{\varphi}_{\tau, h}^{\bm{u}}}_{L^{2}(I_{n}; \bm{U}_{h})}
            & \leq c c_{1} \tau \norm{\partial_{t} \bm{e}_{\bm{u}}}_{L^{2}(I_{n}; \bm{U}_{h})}
            \leq c c_{1} \norm{\bm{e}_{\bm{u}}}_{L^{2}(I_{n}; \bm{U}_{h})}.
        \end{align}
        The usage of \eqref{eq:dg_auxiliary_I-3} results in
        \begin{subequations}
            \label{eq:prf_dg_estimate_tf2-4}
            \begin{align}
                \norm{\bm{\varphi}_{\tau, h}^{\bm{w}}}_{L^{2}(I_{n}; \bm{L}^{2})}
                & = c_{1} \tau \norm{\sum_{i=1}^{k} \beta_{i}  \bm{e}_{\bm{w}}(t_{n, i}^{\text{G}})  L_{n, i}^{\text{G}}}_{L^{2}(I_{n}; \bm{L}^{2})}
                \leq c c_{1} \tau \norm{\bm{e}_{\bm{w}}}_{L^{2}(I_{n}; \bm{L}^{2})}, \\
                \norm{\divv(\bm{\varphi}_{\tau, h}^{\bm{w}})}_{L^{2}(I_{n}; L^{2})}
                & \leq c c_{1} \tau \norm{\divv(\bm{\mathcal{I}}_{\tau}^{\text{G}} \bm{e}_{\bm{w}})}_{L^{2}(I_{n}; L^{2})}.
            \end{align}
        \end{subequations}
        From \eqref{eq:dg_auxiliary_I-3} it follows
        \begin{align}
            \label{eq:prf_dg_estimate_tf2-5}
            \norm{\varphi_{\tau, h}^{p}}_{L^{2}(I_{n}; L^{2})}
            & \leq c c_{1} \tau \norm{e_{p}}_{L^{2}(I_{n}; L^{2})}
            + c_{2} \tau^{2} \norm{\divv(\bm{\mathcal{I}}_{\tau}^{\text{G}} \bm{e}_{\bm{w}})}_{L^{2}(I_{n}; L^{2})}.
        \end{align}
        Inserting \eqref{eq:prf_dg_estimate_tf2-1} followed by \eqref{eq:dg_auxiliary_III-4} and \eqref{eq:dg_auxiliary_III-5} yields
        \begin{align}
            \label{eq:prf_dg_estimate_tf2-6}
            \norm{\bm{\varphi}_{\tau, h}^{\bm{v}}}_{L^{2}(I_{n}; \bm{L}^{2})}
            & \leq c c_{1} \tau \left( \norm{\partial_{t} \bm{e}_{\bm{v}}}_{L^{2}(I_{n}; \bm{L}^{2})}
            + \norm{\partial_{t} \bm{e}_{\bm{w}}}_{L^{2}(I_{n}; \bm{L}^{2})} \right)
            \leq c c_{1} \left( \norm{\bm{e}_{\bm{v}}}_{L^{2}(I_{n}; \bm{L}^{2})}
            + \norm{\bm{e}_{\bm{w}}}_{L^{2}(I_{n}; \bm{L}^{2})} \right).
        \end{align}
        Combining Lemma~\ref{lem:dg_auxiliary_II} and \eqref{eq:dg_auxiliary_I-1} yields
        \begin{align}
            \label{eq:prf_dg_estimate_tf2-7}
            \begin{aligned}
                &  c_{1} \tau \int_{I_{n}} \scalar{\bm{M}_{\rho}
                    \begin{pmatrix}
                        \partial_{t} \bm{e}_{\bm{v}} \\
                        \partial_{t} \bm{e}_{\bm{w}}
                    \end{pmatrix}
                }{
                    \begin{pmatrix}
                        \bm{e}_{\bm{u}, \beta} \\
                        \sum_{i=1}^{k} \beta_{i}  \bm{e}_{\bm{w}}(t_{n, i}^{\text{G}})  L_{n, i}^{\text{G}}
                    \end{pmatrix}
                }
                - \scalar{\partial_{t} \bm{e}_{\bm{u}} - \bm{e}_{\bm{v}}}{\bar{\rho}  \bm{e}_{\bm{v}, \beta} + \rho_{f}  \bm{e}_{\bm{w}, \beta}} \mathrm{d}t \\
                = &  c_{1} \tau \int_{I_{n}} \scalar{\bar{\rho}  \partial_{t} \bm{e}_{\bm{v}} + \rho_{f}  \partial_{t} \bm{e}_{\bm{w}}}{\bm{e}_{\bm{u}, \beta}}
                - \scalar{\bm{e}_{\bm{u}, \beta}}{\bar{\rho}  \partial_{t} \bm{e}_{\bm{v}} + \rho_{f}  \partial_{t} \bm{e}_{\bm{w}}} \mathrm{d}t \\
                & + c_{1} \tau \int_{I_{n}} \scalar{\rho_{f}  \partial_{t} \bm{e}_{\bm{v}} + \rho_{w}  \partial_{t} \bm{e}_{\bm{w}}}{\sum_{i=1}^{k} \beta_{i}  \bm{e}_{\bm{w}}(t_{n, i}^{\text{G}})  L_{n, i}^{\text{G}}}
                + \scalar{\sum_{i=0}^{k} \beta_{i}  \bm{e}_{\bm{v}}(t_{n, i}^{\text{G}, 0})  L_{n, i}^{\text{G}, 0}}{\bar{\rho}  \partial_{t} \bm{e}_{\bm{v}} + \rho_{f}  \partial_{t} \bm{e}_{\bm{w}}} \mathrm{d}t \\
                = &  c_{1} \tau \int_{I_{n}} \scalar{\bm{M}_{\rho}
                    \begin{pmatrix}
                        \partial_{t} \bm{e}_{\bm{v}} \\
                        \partial_{t} \bm{e}_{\bm{w}}
                    \end{pmatrix}
                }{
                    \begin{pmatrix}
                        \sum_{i=0}^{k} \beta_{i}  \bm{e}_{\bm{v}}(t_{n, i}^{\text{G}, 0})  L_{n, i}^{\text{G}, 0} \\
                        \sum_{i=0}^{k} \beta_{i}  \bm{e}_{\bm{w}}(t_{n, i}^{\text{G}, 0})  L_{n, i}^{\text{G}, 0}
                    \end{pmatrix}
                } \\
                \geq &  C_{1} c_{1} \norm{\bm{M}_{\rho}^{1/2}
                    \begin{pmatrix}
                        \bm{e}_{\bm{v}} \\
                        \bm{e}_{\bm{w}}
                    \end{pmatrix}
                }_{L^{2}(I_{n}; \bm{L}^{2})}^{2}
                - C_{2} c_{1} \tau \norm{\bm{M}_{\rho}^{1/2}
                    \begin{pmatrix}
                        \bm{e}_{\bm{v}}(t_{n-1}^{+}) \\
                        \bm{e}_{\bm{w}}(t_{n-1}^{+})
                    \end{pmatrix}
                }_{\bm{L}^{2}(\Omega)}^{2},
            \end{aligned}
        \end{align}
        where we have used \eqref{eq:dg_auxiliary_III-1} in the last step.
        With \eqref{eq:dg_auxiliary_III-2}, we get
        \begin{align}
            \label{eq:prf_dg_estimate_tf2-8}
            \begin{aligned}
                \int_{I_{n}} \scalar{\bm{K}^{-1} \bm{e}_{\bm{w}}}{\bm{\varphi}_{\tau, h}^{\bm{w}}} \mathrm{d}t
                & = c_{1} \tau \int_{I_{n}} \scalar{\bm{K}^{-1} \bm{e}_{\bm{w}}}{\sum_{i=1}^{k} \beta_{i}  \bm{e}_{\bm{w}}(t_{n, i}^{\text{G}})  L_{n, i}^{\text{G}}} \\
                & \geq \tau c_{1} C_{1} \norm{\bm{K}^{-1/2} \bm{e}_{\bm{w}}}_{L^{2}(I_{n}; \bm{L}^{2})}^{2}
                - c_{1} C_{2} \tau^{2} \norm{\bm{K}^{-1/2} \bm{e}_{\bm{w}}(t_{n-1}^{+})}_{\bm{L}^{2}(\Omega)}^{2}
            \end{aligned}
        \end{align}
        From \eqref{eq:dg_auxiliary_I-1} and \eqref{eq:dg_auxiliary_III-1} it follows
        \begin{align}
            \label{eq:prf_dg_estimate_tf2-9}
            \begin{aligned}
                & \quad \int_{I_{n}} \scalar{s_{0}  \partial_{t} e_{p}}{\varphi^{p}_{1}} \mathrm{d}t
                = s_{0} c_{1} \tau \int_{I_{n}} \scalar{\partial_{t} e_{p}}{\sum_{i=1}^{k} \beta_{i}  e_{p}(t_{n, i}^{\text{G}})  L_{n, i}^{\text{G}}} \mathrm{d}t \\
                & = s_{0} c_{1} \tau \int_{I_{n}} \scalar{\partial_{t} e_{p}}{\sum_{i=0}^{k} \beta_{i}  e_{p}(t_{n, i}^{\text{G}, 0})  L_{n, i}^{\text{G}, 0}} \mathrm{d}t
                \geq C_{1} c_{1} s_{0} \norm{e_{p}}_{L^{2}(I_{n}; L^{2})}^{2}
                - C_{2} c_{1} \tau s_{0} \norm{e_{p}(t_{n-1}^{+})}_{L^{2}(\Omega)}^{2}.
            \end{aligned}
        \end{align}
        Moreover, using~\eqref{eq:dg_auxiliary_I-1} and \eqref{eq:dg_auxiliary_II-1} we obtain
        \begin{align}
            \label{eq:prf_dg_estimate_tf2-10}
            \begin{aligned}
                \int_{I_{n}} \scalar{\alpha \divv(\partial_{t} \bm{e}_{\bm{u}})}{\varphi^{p}_{1}} \mathrm{d}t
                & = c_{1} \tau \int_{I_{n}} \scalar{\alpha \divv(\partial_{t} \bm{e}_{\bm{u}})}{\sum_{i=1}^{k} \beta_{i}  e_{p}(t_{n, i}^{\text{G}})  L_{n, i}^{\text{G}}} \mathrm{d}t \\
                & = c_{1} \tau \int_{I_{n}} \scalar{\alpha \divv(\partial_{t} \bm{e}_{\bm{u}})}{\sum_{i=0}^{k} \beta_{i}  e_{p}(t_{n, i}^{\text{G}, 0})  L_{n, i}^{\text{G}, 0}} \mathrm{d}t \\
                & = \int_{I_{n}} \scalar{\alpha  e_{p}}{\divv(c_{1} \tau  \bm{e}_{\bm{u}, \beta})} \mathrm{d}t
                = \int_{I_{n}} \scalar{\alpha  e_{p}}{\divv(\bm{\varphi}_{\tau, h}^{\bm{u}})} \mathrm{d}t.
            \end{aligned}
        \end{align}
        With \eqref{eq:dg_auxiliary_I-2} we infer
        \begin{align}
            \label{eq:prf_dg_estimate_tf2-11}
            \begin{aligned}
                \int_{I_{n}} \scalar{\divv(\bm{e}_{\bm{w}})}{\varphi^{p}_{1}} \mathrm{d}t
                & = c_{1} \tau \int_{I_{n}} \scalar{\divv(\bm{e}_{\bm{w}})}{\sum_{i=1}^{k} \beta_{i}  e_{p}(t_{n, i}^{\text{G}})  L_{n, i}^{\text{G}}} \mathrm{d}t \\
                & = c_{1} \tau \int_{I_{n}} \scalar{e_{p}}{\divv(\sum_{i=1}^{k} \beta_{i}  \bm{e}_{\bm{w}}(t_{n, i}^{\text{G}})  L_{n, i}^{\text{G}})} \mathrm{d}t
                = \int_{I_{n}} \scalar{e_{p}}{\divv(\bm{\varphi}_{\tau, h}^{\bm{w}})} \mathrm{d}t.
            \end{aligned}
        \end{align}
        Inserting \eqref{eq:prf_dg_estimate_tf2-1} and then using Cauchy-Schwarz and Young's inequality as well as \eqref{eq:dg_auxiliary_III-5} yields
        \begin{align}
            \label{eq:prf_dg_estimate_tf2-12}
            \begin{aligned}
                \int_{I_{n}} \scalar{s_{0}  \partial_{t} e_{p}}{\varphi^{p}_{2}} \mathrm{d}t
                & = c_{2} \tau^{2} \int_{I_{n}} \scalar{s_{0}  \partial_{t} e_{p}}{\divv(\bm{\mathcal{I}}_{\tau}^{\text{G}} \bm{e}_{\bm{w}})} \mathrm{d}t \\
                & \geq - c_{2} \tau^{2} \left( s_{0}^2{} \norm{\partial_{t} e_{p}}_{L^{2}(I_{n}; L^{2})}^{2}
                + \frac{1}{4} \norm{\divv(\bm{\mathcal{I}}_{\tau}^{\text{G}} \bm{e}_{\bm{w}})}_{L^{2}(I_{n}; L^{2})}^{2} \right) \\
                & \geq - c_{2} \left( c \norm{e_{p}}_{L^{2}(I_{n}; L^{2})}^{2}
                + \frac{1}{4} \tau^{2} \norm{\divv(\bm{\mathcal{I}}_{\tau}^{\text{G}} \bm{e}_{\bm{w}})}_{L^{2}(I_{n}; L^{2})}^{2} \right).
            \end{aligned}
        \end{align}
        Analogously, we obtain
        \begin{align}
            \label{eq:prf_dg_estimate_tf2-13}
            \begin{aligned}
                \int_{I_{n}} \scalar{\alpha \divv(\partial_{t} \bm{e}_{\bm{u}})}{\varphi^{p}_{2}} \mathrm{d}t
                & = c_{2} \tau^{2} \int_{I_{n}} \scalar{\alpha \divv(\partial_{t} \bm{e}_{\bm{u}})}{\divv(\bm{\mathcal{I}}_{\tau}^{\text{G}} \bm{e}_{\bm{w}})} \mathrm{d}t \\
                & \geq - c_{2} \left( c \norm{\divv(\bm{e}_{\bm{u}})}_{L^{2}(I_{n}; L^{2})}^{2}
                + \frac{1}{4} \tau^{2} \norm{\divv(\bm{\mathcal{I}}_{\tau}^{\text{G}} \bm{e}_{\bm{w}})}_{L^{2}(I_{n}; L^{2})}^{2} \right).
            \end{aligned}
        \end{align}
        Exploiting \eqref{eq:dg_auxiliary_I-1} leads to
        \begin{align}
            \label{eq:prf_dg_estimate_tf2-14}
            \int_{I_{n}} \scalar{\divv(\bm{e}_{\bm{w}})}{\varphi^{p}_{2}} \mathrm{d}t
            & = \int_{I_{n}} \scalar{\sum_{i=1}^{k} \divv(\bm{e}_{\bm{w}}(t_{n, i}^{\text{G}}))  L_{n, i}^{\text{G}}}{\varphi^{p}_{2}} \mathrm{d}t
            = c_{2} \tau^{2} \norm{\divv(\bm{\mathcal{I}}_{\tau}^{\text{G}} \bm{e}_{\bm{w}})}_{L^{2}(I_{n}; L^{2})}^{2}.
        \end{align}
        Applying~\eqref{eq:dg_auxiliary_II-1}, \eqref{eq:dg_auxiliary_III-1} along with the coercivity of $ a_{h}(\cdot, \cdot) $, we get
        \begin{align}
            \label{eq:prf_dg_estimate_tf2-15}
            \begin{aligned}
                \int_{I_{n}} a_{h}(\bm{e}_{\bm{u}}, \bm{\varphi}_{\tau, h}^{\bm{u}}) \mathrm{d}t
                & = \int_{I_{n}} a_{h}(\bm{e}_{\bm{u}}, c_{1} \tau  \bm{e}_{\bm{u}, \beta}) \mathrm{d}t
                = c_{1} \tau \int_{I_{n}} a_{h}(\sum_{i=0}^{k} \beta_{i}  \bm{e}_{\bm{u}}(t_{n, i}^{\text{G}, 0})  L_{n, i}^{\text{G}, 0}, \partial_{t} \bm{e}_{\bm{u}}) \mathrm{d}t \\
                & \geq C_{1} c_{1} \norm{\bm{e}_{\bm{u}}}_{L^{2}(I_{n}; \bm{U}_{h})}^{2}
                - C_{2} c_{1} \tau \norm{\bm{e}_{\bm{u}}(t_{n-1}^{+})}_{\bm{U}_{h}}^{2}.
            \end{aligned}
        \end{align}
        Using \eqref{eq:prf_dg_estimate_tf2-2}--\eqref{eq:prf_dg_estimate_tf2-15}, \eqref{eq:dg_vanishing_term} together with Lemma~\ref{lem:dg_aux_estimates_rhs_I}--Lemma~\ref{lem:dg_estimate_limits} and Young's inequality with $ \varepsilon_{2} > 0 $, we infer \eqref{eq:dg_estimate_tf2}.
    \end{proof}
    
    Ultimately, combining Lemma~\ref{lem:dg_estimate_tf1}--Lemma~\ref{lem:dg_estimate_tf2} provides the following error estimate:
    
    \begin{theorem}
        \label{thm:dg_main_result}
        There holds
        \begin{align*}
            & \quad \norm{\bm{u}_{\tau, h} - \bm{u}}_{L^{\infty}(I; \bm{U}_{h})}^{2}
            + \norm{\bm{M}_{\rho}^{1/2}
                \begin{pmatrix}
                    \bm{v}_{\tau, h} - \bm{v} \\
                    \bm{w}_{\tau, h} - \bm{w}
                \end{pmatrix}
            }_{L^{\infty}(I; \bm{L}^{2})}^{2}
            + \norm{\bm{K}^{-1/2} (\bm{w}_{\tau, h} - \bm{w})}_{L^{\infty}(I; \bm{L}^{2})}^{2}
            + s_{0} \norm{p_{\tau, h} - p}_{L^{\infty}(I; L^{2})}^{2} \\
            & \leq c \left( \tau^{2(k+1)} + h^{2(\ell+1)} \right) \norm{(\bm{u}, \bm{w}, p)}_{I}^{2}
            + c \tau^{2k+2} \norm{\partial_{t}^{k+1}
                \begin{pmatrix}
                    \bm{f} \\
                    \bm{g}
                \end{pmatrix}
            }_{L^{2}(I; \bm{L}^{2})}^{2}.
        \end{align*}
    \end{theorem}
    
    \begin{proof}
        Using Lemma~\ref{lem:dg_estimate_tf1}--Lemma~\ref{lem:dg_estimate_tf2}, we get
        \begin{align*}
            & \quad a_{h}(\bm{e}_{\bm{u}}(t_{n}), \bm{e}_{\bm{u}}(t_{n}))
            - a_{h}(\bm{e}_{\bm{u}}(t_{n-1}), \bm{e}_{\bm{u}}(t_{n-1}))
            + \frac{1}{2} C_{L^{2}} \norm{\bm{e}_{\bm{u}}}_{L^{2}(I_{n}; \bm{U}_{h})}^{2}
            - C_{L^{\infty}} \tau \norm{\bm{e}_{u}(t_{n-1})}_{\bm{U}_{h}}^{2} \\
            & \quad + \frac{1}{2}  \norm{\bm{K}^{-1/2} \bm{e}_{\bm{w}}}_{L^{2}(I_{n}; \bm{L}^{2})}^{2}
            + \norm{\bm{M}_{\rho}^{1/2}
                \begin{pmatrix}
                    \bm{e}_{\bm{v}}(t_{n}) \\
                    \bm{e}_{\bm{w}}(t_{n})
                \end{pmatrix}
            }_{\bm{L}^{2}(\Omega)}^{2}
            - \norm{\bm{M}_{\rho}^{1/2}
                \begin{pmatrix}
                    \bm{e}_{\bm{v}}(t_{n-1}) \\
                    \bm{e}_{\bm{w}}(t_{n-1})
                \end{pmatrix}
            }_{\bm{L}^{2}(\Omega)}^{2} \\
            & \quad + \frac{1}{2} C_{L^{2}} \norm{\bm{M}_{\rho}^{1/2}
                \begin{pmatrix}
                    \bm{e}_{\bm{v}} \\
                    \bm{e}_{\bm{w}}
                \end{pmatrix}
            }_{L^{2}(I_{n}; \bm{L}^{2})}^{2}
            - C_{L^{\infty}} \tau \norm{\bm{M}_{\rho}^{1/2}
                \begin{pmatrix}
                    \bm{e}_{\bm{v}}(t_{n-1}) \\
                    \bm{e}_{\bm{w}}(t_{n-1})
                \end{pmatrix}
            }_{\bm{L}^{2}(\Omega)}^{2} \\
            & \quad + s_{0} \norm{e_{p}(t_{n})}_{L^{2}(\Omega)}^{2}
            - s_{0} \norm{e_{p}(t_{n-1})}_{L^{2}(\Omega)}^{2}
            + \frac{1}{2} C_{L^{2}} s_{0} \norm{e_{p}}_{L^{2}(I_{n}; L^{2})}^{2} - \tau s_{0} C_{L^{\infty}} \norm{\bm{e}_{p}(t_{n-1})}_{L^{2}(\Omega)}^{2} \\
            & \leq c \left( \tau^{2(k+1)} + h^{2(\ell+1)} \right) \norm{(\bm{u}, \bm{w}, p)}_{n}^{2}
            + c\tau^{2k+2} \norm{\partial_{t}^{k+1}
                \begin{pmatrix}
                    \bm{f} \\
                    \bm{g}
                \end{pmatrix}
            }_{L^{2}(I_{n}; \bm{L}^{2})}^{2}.
        \end{align*}
        We sum the above equations on the interval 
        $ \tilde{I}_{n} = \cup_{i=0}^{n} I_{i} $ and 
        consequently, by setting $ \bm{e}_{\bm{u}}(0) = \bm{e}_{\bm{v}}(0) = \bm{e}_{\bm{w}}(0) = e_{p}(0) = 0 $, we obtain
        \begin{align*}
            & \quad \norm{\bm{e}_{\bm{u}}(t_{n})}_{\bm{U}_{h}}^{2}
            + \frac{1}{2} C_{L^{2}} \norm{\bm{e}_{\bm{u}}}_{L^{2}(\tilde{I}_{n}, \bm{U}_{h})}^{2}
            - C_{L^{\infty}} \max_{n} \norm{\bm{e}_{\bm{u}}(t_{n-1})}_{\bm{U}_{h}}^{2}
            + \frac{1}{2} \norm{\bm{K}^{-1/2} \bm{e}_{\bm{w}}}_{L^{2}(\tilde{I}_{n}; \bm{L}^{2})} \\
            & \quad + \norm{\bm{M}_{\rho}^{1/2}
                \begin{pmatrix}
                    \bm{e}_{\bm{v}}(t_{n}) \\
                    \bm{e}_{\bm{w}}(t_{n})
                \end{pmatrix}
            }_{\bm{L}^{2}(\Omega)}^{2}
            + \frac{1}{2} C_{L^{2}} \norm{\bm{M}_{\rho}^{1/2}
                \begin{pmatrix}
                    \bm{e}_{\bm{v}} \\
                    \bm{e}_{\bm{w}}
                \end{pmatrix}
            }_{L^{2}(\tilde{I}_{n}; \bm{L}^{2})}^{2}
            - C_{L^{\infty}} \max_{n} \norm{\bm{M}_{\rho}^{1/2}
                \begin{pmatrix}
                    \bm{e}_{\bm{v}}(t_{n-1}) \\
                    \bm{e}_{\bm{w}}(t_{n-1})
                \end{pmatrix}
            }_{\bm{L}^{2}(\Omega)}^{2} \\
            & \quad + s_{0} \norm{e_{p}(t_{n})}_{L^{2}(\Omega)}^{2}
            + \frac{1}{2} C_{L^{2}} s_{0} \norm{e_{p}}_{L^{2}(\tilde{I}_{n}; L^{2})}^{2}
            - C_{L^{\infty}} \max_{n} s_{0} \norm{e_{p}(t_{n-1})}_{L^{2}(\Omega)}^{2} \\
            & \leq c \left( \tau^{2(k+1)} + h^{2(\ell+1)} \right) \norm{(\bm{u}, \bm{w}, p)}_{I}^{2}
            + c\tau^{2k+2} \norm{\partial_{t}^{k+1}
                \begin{pmatrix}
                    \bm{f} \\
                    \bm{g}
                \end{pmatrix}
            }_{L^{2}(I; \bm{L}^{2})}^{2}.
        \end{align*}
        Now we choose $ C_{L^{\infty}} \leq \frac{1}{2} $ in order to get
        \begin{align*}
            & \quad \norm{\bm{e}_{\bm{u}}}_{L^{2}(\tilde{I}_{n}; \bm{U}_{h})}^{2}
            + \max_{n} \norm{\bm{e}_{\bm{u}}(t_{n-1})}_{\bm{U}_{h}}^{2}
            + \norm{\bm{K}^{-1/2} \bm{e}_{\bm{w}}}_{L^{2}(\tilde{I}_{n}; \bm{L}^{2})}^{2}
            + \norm{\bm{M}_{\rho}^{1/2}
                \begin{pmatrix}
                    \bm{e}_{\bm{v}} \\
                    \bm{e}_{\bm{w}}
                \end{pmatrix}
            }_{L^{2}(\tilde{I}_{n}; \bm{L}^{2})}^{2} \\
            & \quad + \max_{n} \norm{\bm{M}_{\rho}^{1/2}
                \begin{pmatrix}
                    \bm{e}_{\bm{v}}(t_{n-1}) \\
                    \bm{e}_{\bm{v}}(t_{n-1})
                \end{pmatrix}
            }_{\bm{L}^{2}(\Omega)}^{2}
            + s_{0} \norm{e_{p}}_{L^{2}(\Omega)}^{2}
            + \max_{n} s_{0} \norm{e_{p}(t_{n-1})}_{L^{2}(\Omega)}^{2} \\
            & \leq c \left( \tau^{2(k+1)} + h^{2(\ell+1)} \right) \norm{(\bm{u}, \bm{w}, p)}_{I}^{2}
            + c\tau^{2k+2} \norm{\partial_{t}^{k+1}
                \begin{pmatrix}
                    \bm{f} \\
                    \bm{g}
                \end{pmatrix}
            }_{L^{2}(I; \bm{L}^{2})}^{2}.
        \end{align*}
        Thus, we have
        \begin{align*}
            & \quad \max_{n} \norm{\bm{e}_{\bm{u}}(t_{n})}_{\bm{U}_{h}}^{2}
            + \norm{\bm{K}^{-1/2} \bm{e}_{\bm{w}}}_{L^{2}(\tilde{I}_{n}; \bm{L}^{2})}^{2}
            + \max_{n} \norm{\bm{M}_{\rho}^{1/2}
                \begin{pmatrix}
                    \bm{e}_{\bm{v}}(t_{n}) \\
                    \bm{e}_{\bm{w}}(t_{n})
                \end{pmatrix}
            }_{\bm{L}^{2}(\Omega)}^{2}
            + \max_{n} s_{0} \norm{e_{p}(t_{n})}_{L^{2}(\Omega)}^{2} \\
            & \leq c \left( \tau^{2(k+1)} + h^{2(\ell+1)} \right) \norm{(\bm{u}, \bm{w}, p)}_{I}^{2}
            + c\tau^{2k+2} \norm{\partial_{t}^{k+1}
                \begin{pmatrix}
                    \bm{f} \\
                    \bm{g}
                \end{pmatrix}
            }_{L^{2}(I; \bm{L}^{2})}^{2}.
        \end{align*}
        From Lemma~\ref{lem:dg_estimate_tf2}, we obtain
        \begin{align*}
            & \quad C_{L^{2}} \left( \norm{\bm{e}_{\bm{u}}}_{L^{2}(I_{n}; \bm{U}_{h})}^{2}
            + \norm{\bm{M}_{\rho}^{1/2}
                \begin{pmatrix}
                    \bm{e}_{\bm{v}} \\
                    \bm{e}_{\bm{w}}
                \end{pmatrix}
            }_{L^{2}(I_{n}; \bm{L}^{2})}^{2}
            + s_{0} \norm{e_{p}}_{L^{2}(I_{n}; L^{2})}^{2} \right) \\
            & \leq \tau c \left( \norm{\bm{e}_{\bm{u}}(t_{n-1})}_{\bm{U}_{h}}^{2}
            + \norm{\bm{M}_{\rho}^{1/2}
                \begin{pmatrix}
                    \bm{e}_{\bm{v}}(t_{n-1}) \\
                    \bm{e}_{\bm{w}}(t_{n-1})
                \end{pmatrix}
            }_{\bm{L}^{2}(\Omega)}^{2}
            + \norm{\bm{K}^{-1/2} \bm{e}_{\bm{w}}}_{L^{2}(I_{n}; \bm{L}^{2})}^{2}
            + s_{0} \norm{e_{p}(t_{n-1})}_{L^{2}(\Omega)}^{2} \right) \\
            & \quad + c \tau \left( \tau^{2(k+1)} + h^{2(\ell+1)} \right) \norm{(\bm{u}, \bm{w}, p)}_{n}^{2}
            + c\tau^{2k+3} \norm{\partial_{t}^{k+1}
                \begin{pmatrix}
                    \bm{f} \\
                    \bm{g}
                \end{pmatrix}
            }_{L^{2}(I_{n}; \bm{L}^{2})}^{2}.
        \end{align*}
        The $ L^{\infty} $--$ L^{2} $ inverse property, cf. \cite{Karakashian2004Conv} gives
        \begin{align*}
            & \quad \norm{\bm{e}_{\bm{u}}}_{L^{\infty}(I_{n}; \bm{U}_{h})}^{2}
            + \norm{\bm{M}_{\rho}^{1/2}
                \begin{pmatrix}
                    \bm{e}_{\bm{v}} \\
                    \bm{e}_{\bm{w}}
                \end{pmatrix}
            }_{L^{\infty}(I_{n}; \bm{L}^{2})}^{2}
            + s_{0} \norm{e_{p}}_{L^{\infty}(I_{n}; L^{2})}^{2} \\
            & \leq c \left(  \norm{\bm{e}_{\bm{u}}(t_{n-1})}_{\bm{U}_{h}}^{2}
            + \norm{\bm{M}_{\rho}^{1/2}
                \begin{pmatrix}
                    \bm{e}_{\bm{v}}(t_{n-1}) \\
                    \bm{e}_{\bm{w}}(t_{n-1}) \\
                \end{pmatrix}
            }_{\bm{L}^{2}(\Omega)}^{2}
            + \norm{\bm{K}^{-1/2} \bm{e}_{\bm{w}}}_{L^{2}(I_{n}; \bm{L}^{2})}^{2}
            + s_{0} \norm{e_{p}(t_{n-1})}_{L^{2}(\Omega)}^{2} \right) \\
            & \quad + c \left( \tau^{2(k+1)} + h^{2(\ell+1)} \right) \norm{(\bm{u}, \bm{w}, p)}_{n}^{2}
            + \tau^{2k+2} \norm{\partial_{t}^{k+1}
                \begin{pmatrix}
                    \bm{f} \\
                    \bm{g}
                \end{pmatrix}
            }_{L^{2}(I_{n}; \bm{L}^{2})}^{2} \\
            & \leq c \left( \tau^{2(k+1)} + h^{2(\ell+1)} \right) \norm{(\bm{u}, \bm{w}, p)}_{n}^{2}
            + c\tau^{2k+2} \norm{\partial_{t}^{k+1}
                \begin{pmatrix}
                    \bm{f} \\
                    \bm{g}
                \end{pmatrix}
            }_{L^{2}(I_{n}; \bm{L}^{2})}^{2}.
        \end{align*}
        Using \eqref{eq:dg_splitting} and Lemma \ref{lem:dg_estimates_eta}, it follows the result.
    \end{proof}
    
    For a convex domain, we can also prove an error estimate of optimal order in $ h $ for the displacement $ \bm{u} $, cf.~\cite{Arnold2001Unif}:
    
    \begin{corollary}
        If $ \Omega $ is convex, the following optimal order estimate holds for $ \bm{u} $
        \begin{align*}
            \norm{\bm{u}_{\tau, h} - \bm{u}}_{L^{\infty}(I; \bm{L}^{2})}^{2}
            & \leq c \left( \tau^{2(k+1)} + h^{2(\ell+2)} \right) \norm{(\bm{u}, \bm{w}, p)}_{I}^{2}
            + c\tau^{2k+2} \norm{\partial_{t}^{k+1}
                \begin{pmatrix}
                    \bm{f} \\
                    \bm{g}
                \end{pmatrix}
            }_{L^{2}(I; \bm{L}^{2})}^{2}.
        \end{align*}
    \end{corollary}
    
    \begin{proof}
        The estimate follows from Theorem~\ref{thm:dg_main_result} under full elliptic regularity.
    \end{proof}


    \section*{Acknowledgement}
    The first, second and third authors acknowledge the funding by the German Science Fund (Deutsche Forschungsgemeinschaft, DFG) as part of the project ``Physics-oriented solvers for multicompartmental poromechanics'' under grant number 456235063.
    
    \printbibliography
    
    \appendix
    \section{Proof of Lemma~\ref{lem:dg_auxiliary_I}}
    \label{sec:proof_auxiliary_I}
    
    We note that $ L_{n, i}^{\text{G}, 0}(t_{n, j}^{\text{G}}) = \delta_{ij} $ for $ i = 0, \ldots, k $ and $ j = 1, \ldots, k $ as well as $ L_{n, i}^{\text{G}}(t_{n, j}^{\text{G}}) = \delta_{ij} $ for $ i, j = 1, \ldots, k $. Using the exactness of \eqref{eq:quadrature_g} for polynomials of degree $ 2k-1 $ yields
    \begin{align*}
        \int_{I_{n}} (\sum_{i=0}^{k} \beta_{i} x_{i}  L_{n, i}^{\text{G}, 0}, z) \mathrm{d}t
        & = \tau \sum_{l=1}^{k} \sum_{i=0}^{k} \hat{\omega}_{l}^{\text{G}}  (\beta_{i} x_{i}, z(t_{n, l}^{\text{G}}))  L_{n, i}^{\text{G}, 0}(t_{n, l}^{\text{G}})
        = \tau \sum_{i=1}^{k} \hat{\omega}_{i}^{\text{G}}  (\beta_{i} x_{i}, z(t_{n, i}^{\text{G}})) \\
        & = \tau \sum_{l=1}^{k} \sum_{i=1}^{k} \hat{\omega}_{l}^{\text{G}}  (\beta_{i} x_{i}, z(t_{n, l}^{\text{G}}))  L_{n, i}^{\text{G}}(t_{n, l}^{\text{G}})
        = \int_{I_{n}} (\sum_{i=1}^{k} \beta_{i} x_{i}  L_{n, i}^{\text{G}}, z) \mathrm{d}t.
    \end{align*}
    Thus, we have \eqref{eq:dg_auxiliary_I-1}.
    By the exactness of the Gauss quadrature rule~\eqref{eq:quadrature_g}, we deduce
    \begin{align}
        \label{eq:prf_dg_auxiliary_I-1}
        \int_{I_{n}} L_{n, j}^{\text{G}, 0}  L_{n, i}^{\text{G}} \mathrm{d}t
        & = \tau \sum_{l=1}^{k} \hat{\omega}_{l}^{\text{G}}  L_{n, j}^{\text{G}, 0}(t_{n, l}^{\text{G}})  L_{n, i}^{\text{G}}(t_{n, l}^{\text{G}})
        = \tau \hat{\omega}_{i}^{\text{G}}  \delta_{ij}, \quad i = 1, \ldots, k, j = 0, \ldots, k.
    \end{align}
    Consequently, we receive
    \begin{align*}
        \int_{I_{n}} (x, \sum_{i=1}^{k} \beta_{i} y_{i}  L_{n, i}^{\text{G}}) \mathrm{d}t
        & = \sum_{i=1}^{k} \sum_{j=0}^{k} \beta_{i}  (x_{j}, y_{i}) \int_{I_{n}} L_{n, j}^{\text{G}, 0}  L_{n, i}^{\text{G}} \mathrm{d}t
        = \tau \sum_{i=1}^{k} \hat{\omega}_{i}^{\text{G}} \beta_{i}  (x_{i}, y_{i}) \\
        & = \sum_{i=0}^{k} \sum_{j=1}^{k} \beta_{j}  (x_{j}, y_{i}) \int_{I_{n}} L_{n, j}^{\text{G}}  L_{n, i}^{\text{G}, 0} \mathrm{d}t
        = \int_{I_{n}} (\sum_{j=1}^{k} \beta_{j} x_{j}  L_{n, j}^{\text{G}}, y) \mathrm{d}t,
    \end{align*}
    which shows \eqref{eq:dg_auxiliary_I-2}.
    Similarly, the exactness of \eqref{eq:quadrature_g} leads to the orthogonality
    \begin{align}
        \label{eq:prf_dg_auxiliary_I-2}
        \int_{I_{n}} L_{n, j}^{\text{G}}  L_{n, i}^{\text{G}} \mathrm{d}t
        & = \tau \sum_{l=1}^{k} \hat{\omega}_{l}^{\text{G}}  L_{n, j}^{\text{G}}(t_{n, l}^{\text{G}})  L_{n, i}^{\text{G}}(t_{n, l}^{\text{G}})
        = \tau \hat{\omega}_{i}^{\text{G}}  \delta_{ij}.
    \end{align}
    Exploiting this orthogonality together with the equivalence
    \begin{align}
        \label{eq:prf_dg_auxiliary_I-3}
        C_{1} \left( \tau \sum_{i=0}^{k} \norm{x_{i}}_{L^{2}(\Omega)}^{2} \right)^{1/2}
        \leq \norm{x}_{L^{2}(I_{n}; L^{2})}
        \leq C_{2} \left( \tau \sum_{i=0}^{k} \norm{x_{i}}_{L^{2}(\Omega)}^{2} \right)^{1/2},
    \end{align}
    see \cite[Eq.~(2.4)]{Karakashian2004Conv}, we end up with
    \begin{align*}
        \norm{\sum_{i=1}^{k} \beta_{i} x_{i}  L_{n, i}^{\text{G}}}_{L^{2}(I_{n}; L^{2})}^{2}
        & = \sum_{i, j=1}^{k} \scalar{\beta_{j} x_{j}}{\beta_{i} x_{i}} \int_{I_{n}} L_{n, j}^{\text{G}}  L_{n, i}^{\text{G}} \mathrm{d}t
        = \tau \sum_{i=1}^{k} \hat{\omega}_{i}^{\text{G}} \beta_{i}^{2}  \norm{x_{i}}_{L^{2}(\Omega)}^{2} \\
        & \leq \max_{m=1, \ldots, k} \{\hat{\omega}_{m}^{\text{G}} \beta_{m}^{2}\}  \tau \sum_{i=0}^{k} \norm{x_{i}}_{L^{2}(\Omega)}^{2}
        \leq c \norm{x}_{L^{2}(I_{n}; L^{2})}^{2},
    \end{align*}
    which confirms \eqref{eq:dg_auxiliary_I-3}.
    Assertion~\eqref{eq:dg_auxiliary_I-5} is a special case of \eqref{eq:dg_auxiliary_I-3} using $ \beta_{i} = 1 $ for $ i = 1, \ldots, k $.

    \section{Proof of Lemma~\ref{lem:dg_auxiliary_II}}
    \label{sec:proof_auxiliary_II}
    
    Using the exactness of \eqref{eq:quadrature_g} for polynomials up to degree $ 2k-1 $, we get for $ i = 0, \ldots, k $
    \begin{align}
        \label{eq:prf_dg_auxiliary_II-1}
        \int_{I_{n}} \partial_{t} L_{n, i}^{\text{G}, 0}  L_{n, j}^{\text{G}, 0} \mathrm{d}t
        & = \tau \sum_{l=1}^{k} \hat{\omega}_{l}^{\text{G}}  \partial_{t} L_{n, i}^{\text{G}, 0}(t_{n, l}^{\text{G}})  L_{n, j}^{\text{G}, 0}(t_{n, l}^{\text{G}})
        =
        \begin{cases}
            \tau \hat{\omega}_{j}^{\text{G}}  \partial_{t} L_{n, i}^{\text{G}, 0}(t_{n, j}^{\text{G}}), & \text{for $ j = 1, \ldots, k $}, \\
            0, & \text{for $ j = 0 $},
        \end{cases}
    \end{align}
    since $ t_{n, i}^{\text{G}, 0} = t_{n, i}^{\text{G}} $ for $ i = 1, \ldots, k $ and $ L_{n, i}^{\text{G}, 0}(t_{n, j}^{\text{G}, 0}) = \delta_{ij} $ for $ i, j = 0, \ldots, k $ as defined in Section~\ref{subsec:quadrature_formulas}.
    Utilizing \eqref{eq:dg_auxiliary_II_assumption} followed by \eqref{eq:prf_dg_auxiliary_I-1} and \eqref{eq:prf_dg_auxiliary_II-1}, we can infer that
    \begin{align*}
        \int_{I_{n}} (x, y_{\beta}) \mathrm{d}t
        & = \sum_{m, i=0}^{k} \sum_{j=1}^{k} \beta_{j}  \partial_{t} L_{n, i}^{\text{G}, 0}(t_{n, j}^{\text{G}})  (x_{m}, y_{i}) \int_{I_{n}} L_{n, m}^{\text{G}, 0}  L_{n, j}^{\text{G}} \mathrm{d}t
        = \tau \sum_{i=0}^{k} \sum_{j=1}^{k} \hat{\omega}_{j}^{\text{G}} \beta_{j}  \partial_{t} L_{n, i}^{\text{G}, 0}(t_{n, j}^{\text{G}})  (x_{j}, y_{i}) \\
        & = \sum_{i, j=0}^{k} \beta_{j}  (x_{j}, y_{i}) \int_{I_{n}} L_{n, j}^{\text{G}, 0}  \partial_{t} L_{n, i}^{\text{G}, 0} \mathrm{d}t
        = \int_{I_{n}} (\sum_{j=0}^{k} \beta_{j} x_{j}  L_{n, j}^{\text{G}, 0}, \partial_{t} y) \mathrm{d}t,
    \end{align*}
    hence we have \eqref{eq:dg_auxiliary_II-1}.
    In addition, the exactness of \eqref{eq:quadrature_g} yields
    \begin{align}
        \label{eq:prf_dg_auxiliary_II-2}
        \int_{I_{n}} \partial_{t} L_{n, i}^{\text{G}, 0}  L_{n, j}^{\text{G}} \mathrm{d}t
        & = \tau \sum_{l=1}^{k} \hat{\omega}_{l}^{\text{G}}  \partial_{t} L_{n, i}^{\text{G}, 0}(t_{n, l}^{\text{G}})  L_{n, j}^{\text{G}}(t_{n, l}^{\text{G}})
        = \tau \hat{\omega}_{j}^{\text{G}}  \partial_{t} L_{n, i}^{\text{G}, 0}(t_{n, j}^{\text{G}}), \quad i = 0, \ldots, k, j = 1, \ldots, k,
    \end{align}
    where we have used $ L_{n, i}^{\text{G}}(t_{n, j}^{\text{G}}) = \delta_{ij} $ for $ i, j = 1, \ldots, k $.
    Employing \eqref{eq:dg_auxiliary_II_assumption} and \eqref{eq:prf_dg_auxiliary_II-2}, we get
    \begin{align*}
        \int_{I_{n}} (\partial_{t} x, y_{\beta}) \mathrm{d}t
        & = \sum_{m, i=0}^{k} \sum_{j=1}^{k} \beta_{j}  \partial_{t} L_{n, i}^{\text{G}, 0}(t_{n, j}^{\text{G}})  (x_{m}, y_{i}) \int_{I_{n}} \partial_{t} L_{n, m}^{\text{G}, 0}  L_{n, j}^{\text{G}} \mathrm{d}t \\
        & = \tau \sum_{m, i=0}^{k} \sum_{j=1}^{k} \hat{\omega}_{j}^{\text{G}} \beta_{j}  \partial_{t} L_{n, i}^{\text{G}, 0}(t_{n, j}^{\text{G}})  \partial_{t} L_{n, m}^{\text{G}, 0}(t_{n, j}^{\text{G}})  (x_{m}, y_{i}) \\
        & = \sum_{m, i=0}^{k} \sum_{j=1}^{k} \beta_{j}  \partial_{t} L_{n, m}^{\text{G}, 0}(t_{n, j}^{\text{G}})  (x_{m}, y_{i}) \int_{I_{n}} L_{n, j}^{\text{G}}  \partial_{t} L_{n, i}^{\text{G}, 0} \mathrm{d}t
        = \int_{I_{n}} (x_{\beta}, \partial_{t} y) \mathrm{d}t.
    \end{align*}

    \section{Proof of Lemma~\ref{lem:dg_auxiliary_III}}
    \label{sec:proof_auxiliary_III}
    
    Using \eqref{eq:dg_auxiliary_III_assumption} together with \eqref{eq:dg_auxiliary_I-1}, we obtain
    \begin{align}
        \label{eq:prf_dg_auxiliary_III-1}
        \begin{aligned}
            \int_{I_{n}} \scalar{\sum_{i=0}^{k} \beta_{i} x_{i}  L_{n, i}^{\text{G}, 0}}{\partial_{t} x} \mathrm{d}t
            & = \int_{I_{n}} \scalar{\sum_{i=1}^{k} \beta_{i} x_{i}  L_{n, i}^{\text{G}}}{\partial_{t} x} \mathrm{d}t
            = \sum_{i=1}^{k} \sum_{j=0}^{k} \beta_{i} \scalar{x_{i}}{x_{j}} \int_{I_{n}} L_{n, i}^{\text{G}}  \partial_{t} L_{n, j}^{\text{G}, 0} \mathrm{d}t \\
            & = \sum_{i, j=1}^{k} \beta_{i}^{1/2}  m_{ij}  \beta_{j}^{-1/2} \scalar{\beta_{i}^{1/2} x_{i}}{\beta_{j}^{1/2} x_{j}}
            + \sum_{i=1}^{k} \beta_{i}  m_{i0} \scalar{x_{i}}{x_{0}},
        \end{aligned}
    \end{align}
    where
    \begin{align*}
        m_{ij} & := \int_{I_{n}} L_{n, i}^{\text{G}}  \partial_{t} L_{n, j}^{\text{G}, 0} \mathrm{d}t, \qquad i = 1, \ldots, k, j = 0, \ldots, k.
    \end{align*}
    Next, we define $ \bm{M}, \bm{\widetilde{M}}, \bm{D} \in \mathbb{R}^{k \times k} $ by
    \begin{align*}
        \bm{M} & := (m_{ij})_{i,j = 1}^{k}, &
        \bm{\widetilde{M}} & := \bm{D}^{-1/2} \bm{M} \bm{D}^{1/2}, &
        \bm{D} := \diag(\hat{t}_{1}^{\text{G}}, \ldots, \hat{t}_{k}^{\text{G}}).
    \end{align*}
    Exploiting $ \beta_{i} = (\hat{t}_{i}^{\text{G}})^{-1} $ for $ i = 1, \ldots, k $ and the positive definiteness of $ \bm{\widetilde{M}} $, cf.~\cite[Lem.~2.1]{Karakashian1999ASpa}, we conclude
    \begin{align}
        \label{eq:prf_dg_auxiliary_III-2}
        \sum_{i, j=1}^{k} \beta_{i}^{1/2}  m_{ij}  \beta_{j}^{-1/2} \scalar{\beta_{i}^{1/2} x_{i}}{\beta_{j}^{1/2} x_{j}}
        = \sum_{i, j=1}^{k} \widetilde{m}_{ij} \scalar{\beta_{i}^{1/2} x_{i}}{\beta_{j}^{1/2} x_{j}}
        \geq c \sum_{i=1}^{k} \scalar{\beta_{i}^{1/2} x_{i}}{\beta_{i}^{1/2} x_{i}}
        \geq c \sum_{i=1}^{k} \norm{x_{i}}_{L^{2}(\Omega)}^{2}.
    \end{align}
    
    We note that $ L_{n, i}^{\text{G}} = \hat{L}_{i}^{\text{G}} \circ T_{n}^{-1} $ for $ i = 1, \ldots, k $ and $ L_{n, j}^{\text{G}, 0} = \hat{L}_{j}^{\text{G}, 0} \circ T_{n}^{-1} $ for $ j = 0, \ldots, k $, where $ \{\hat{L}_{i}^{\text{G}}\}_{i=1}^{k} $ and $ \{\hat{L}_{j}^{\text{G}, 0}\}_{j=0}^{k} $ represent the corresponding interpolation polynomials on $ \hat{I} $.
    Furthermore, we have
    \begin{align}
        \label{eq:prf_dg_auxiliary_III-3}
        \norm{L_{n, i}^{\text{G}}}_{L^{2}(I_{n})}^{2} = \tau \norm{\hat{L}_{i}^{\text{G}}}_{L^{2}(\hat{I})}^{2}
        \qquad \text{and} \qquad
        \norm{\partial_{t} L_{n, j}^{\text{G}, 0}}_{L^{2}(I_{n})}^{2} = \tau^{-1} \norm{\partial_{\hat{t}} \hat{L}_{j}^{\text{G}, 0}}_{L^{2}(\hat{I})}^{2}.
    \end{align}
    Using the Cauchy-Schwarz inequality and Young's inequality with $ \varepsilon > 0 $, it follows
    \begin{align}
        \label{eq:prf_dg_auxiliary_III-4}
        \begin{aligned}
            \sum_{i=1}^{k} \beta_{i}  m_{i0} \scalar{x_{i}}{x_{0}}
            & = \int_{I_{n}} \scalar{\sum_{i=1}^{k} \beta_{i} x_{i}  L_{n, i}^{\text{G}}}{x_{0}  \partial_{t} L_{n, 0}^{\text{G}, 0}} \mathrm{d}t
            \geq - \norm{\sum_{i=1}^{k} \beta_{i} x_{i}  L_{n, i}^{\text{G}}}_{L^{2}(I_{n}; L^{2})}  \norm{x_{0}  \partial_{t} L_{n, 0}^{\text{G}, 0}}_{L^{2}(I_{n}; L^{2})} \\
            & \geq - c \sum_{i=1}^{k} \norm{x_{i}}_{L^{2}(\Omega)}  \norm{L_{n, i}^{\text{G}}}_{L^{2}(I_{n})}  \norm{x_{0}}_{L^{2}(\Omega)}  \norm{\partial_{t} L_{n, 0}^{\text{G}, 0}}_{L^{2}(I_{n})} \\
            & \geq - c \left( \sum_{i=1}^{k} \norm{x_{i}}_{L^{2}(\Omega)}^{2} \right)^{1/2} \norm{x_{0}}_{L^{2}(\Omega)}
            \geq - \frac{c}{2} \varepsilon \sum_{i=1}^{k} \norm{x_{i}}_{L^{2}(\Omega)}^{2}
            - \frac{c}{2 \varepsilon} \norm{x_{0}}_{L^{2}(\Omega)}^{2},
        \end{aligned}
    \end{align}
    where $ x_{0} = x(t_{n-1}^{+}) $.
    Combining \eqref{eq:prf_dg_auxiliary_III-1}, \eqref{eq:prf_dg_auxiliary_III-2} and \eqref{eq:prf_dg_auxiliary_III-4} with $ \varepsilon > 0 $ small enough, we receive
    \begin{align*}
        \int_{I_{n}} \scalar{\sum_{i=0}^{k} \beta_{i} x_{i}  L_{n, i}^{\text{G}, 0}}{\partial_{t} x} \mathrm{d}t
        & \geq C_{1} \sum_{i=1}^{k} \norm{x_{i}}_{L^{2}(\Omega)}^{2}
        - C_{2} \norm{x_{0}}_{L^{2}(\Omega)}^{2}
        = C_{1} \sum_{i=0}^{k} \norm{x_{i}}_{L^{2}(\Omega)}^{2}
        - (C_{1} + C_{2}) \norm{x_{0}}_{L^{2}(\Omega)}^{2}.
    \end{align*}
    Multiplying this inequality by $ \tau $ and utilizing \eqref{eq:prf_dg_auxiliary_I-3} results in
    \begin{align*}
        \tau \int_{I_{n}} \scalar{\sum_{i=0}^{k} \beta_{i} x_{i}  L_{n, i}^{\text{G}, 0}}{\partial_{t} x} \mathrm{d}t
        & \geq C_{1, 1} \norm{x}_{L^{2}(I_{n}; L^{2})}^{2}
        - C_{2, 1} \tau \norm{x_{0}}_{L^{2}(\Omega)}^{2},
    \end{align*}
    showing \eqref{eq:dg_auxiliary_III-1}.
    Using equations \eqref{eq:prf_dg_auxiliary_I-1} and \eqref{eq:prf_dg_auxiliary_I-3}, we infer
    \begin{align*}
        \int_{I_{n}} \scalar{\sum_{i=1}^{k} \beta_{i} x_{i}  L_{n, i}^{\text{G}}}{x} \mathrm{d}t
        & = \sum_{i=1}^{k} \sum_{j=0}^{k} \beta_{i} \scalar{x_{i}}{x_{j}} \int_{I_{n}} L_{n, i}^{\text{G}}  L_{n, j}^{\text{G}, 0} \mathrm{d}t
        = \tau \sum_{i=1}^{k} \hat{\omega}_{i}^{\text{G}} \beta_{i} \scalar{x_{i}}{x_{i}} \\
        & \geq c \tau \sum_{i=0}^{k} \norm{x_{i}}_{L^{2}(\Omega)}^{2}
        - c \tau \norm{x_{0}}_{L^{2}(\Omega)}^{2}
        \geq C_{1, 2} \norm{x}_{L^{2}(I_{n}; L^{2})}^{2}
        - C_{2, 2} \tau \norm{x_{0}}_{L^{2}(\Omega)}^{2},
    \end{align*}
    which confirms \eqref{eq:dg_auxiliary_III-2} with $ C_{1} = \min\{C_{1, 1}, C_{1, 2}\} $, $ C_{2} = \max\{C_{2, 1}, C_{2, 2}\} $.
    Inserting \eqref{eq:dg_auxiliary_III_assumption} along with \eqref{eq:prf_dg_auxiliary_II-2} and the exactness of \eqref{eq:quadrature_g}, leads to
    \begin{align*}
        \int_{I_{n}} \scalar{\partial_{t} x}{x_{\beta}} \mathrm{d}t
        & = \sum_{m=0}^{k} \sum_{i=0}^{k} \sum_{j=1}^{k} \beta_{j}  \partial_{t} L_{n, i}^{\text{G}, 0}(t_{n, j}^{\text{G}}) \scalar{x_{m}}{x_{i}} \int_{I_{n}} \partial_{t} L_{n, m}^{\text{G}, 0}  L_{n, j}^{\text{G}} \mathrm{d}t \\
        & = \tau \sum_{j=1}^{k} \hat{\omega}_{j}^{\text{G}} \beta_{j} \scalar{\sum_{m=0}^{k} x_{m}  \partial_{t} L_{n, m}^{\text{G}, 0}(t_{n, j}^{\text{G}})}{\sum_{i=0}^{k} x_{i}  \partial_{t} L_{n, i}^{\text{G}, 0}(t_{n, j}^{\text{G}})}
        \leq \max_{j=1, \ldots, k} \beta_{j} \int_{I_{n}} \scalar{\partial_{t} x}{\partial_{t} x} \mathrm{d}t.
    \end{align*}
    Analogously, it follows
    \begin{align*}
        \int_{I_{n}} \scalar{\partial_{t} x}{x_{\beta}} \mathrm{d}t
        & \geq \min_{j=1, \ldots, k} \beta_{j} \int_{I_{n}} \scalar{\partial_{t} x}{\partial_{t} x} \mathrm{d}t.
    \end{align*}
    Utilizing \eqref{eq:dg_auxiliary_III-3} coupled with the Cauchy-Schwarz inequality, we deduce
    \begin{align*}
        c \norm{\partial_{t} x}_{L^{2}(I_{n}; L^{2})}^{2}
        & \leq \int_{I_{n}} \scalar{\partial_{t} x}{x_{\beta}} \mathrm{d}t
        \leq \norm{\partial_{t} x}_{L^{2}(I_{n}; L^{2})}  \norm{x_{\beta}}_{L^{2}(I_{n}; L^{2})},
    \end{align*}
    thus it holds
    \begin{align}
        \label{eq:prf_dg_auxiliary_III-5}
        c \norm{\partial_{t} x}_{L^{2}(I_{n}; L^{2})} \leq \norm{x_{\beta}}_{L^{2}(I_{n}; L^{2})}.
    \end{align}
    The opposite inequality can be shown with \eqref{eq:prf_dg_auxiliary_I-2} and the fact that the quadrature rule~\eqref{eq:quadrature_g} is exact for polynomials up to degree $ 2k-1 $, thence we become
    \begin{align}
        \label{eq:prf_dg_auxiliary_III-6}
        \begin{aligned}
            \norm{x_{\beta}}_{L^{2}(I_{n}; L^{2})}^{2}
            & = \sum_{i, l=0}^{k} \sum_{j, m=1}^{k} \beta_{m} \beta_{j}  \partial_{t} L_{n, l}^{\text{G}, 0}(t_{n, m}^{\text{G}})  \partial_{t} L_{n, i}^{\text{G}, 0}(t_{n, j}^{\text{G}}) \scalar{x_{l}}{x_{i}} \int_{I_{n}} L_{n, m}^{\text{G}}  L_{n, j}^{\text{G}} \mathrm{d}t \\
            & = \tau \sum_{j=1}^{k} \hat{\omega}_{j}^{\text{G}} \beta_{j}^{2} \scalar{\sum_{l=0}^{k} x_{l}  \partial_{t} L_{n, l}^{\text{G}, 0}(t_{n, j}^{\text{G}})}{\sum_{i=0}^{k} x_{i}  \partial_{t} L_{n, i}^{\text{G}, 0}(t_{n, j}^{\text{G}})}
            \leq \max_{j=1, \ldots, k} \beta_{j}^{2} \int_{I_{n}} \scalar{\partial_{t} x}{\partial_{t} x} \mathrm{d}t.
        \end{aligned}
    \end{align}
    Together, equations \eqref{eq:prf_dg_auxiliary_III-5} and \eqref{eq:prf_dg_auxiliary_III-6} yield \eqref{eq:dg_auxiliary_III-4}.
    We recall that $ L_{n, i}^{\text{G}, 0} = \hat{L}_{i}^{\text{G}, 0} \circ T_{n}^{-1} $.
    Using the triangle inequality, the Cauchy-Schwarz inequality, \eqref{eq:prf_dg_auxiliary_III-3} and \eqref{eq:prf_dg_auxiliary_I-3}, we get
    \begin{align*}
        \norm{\partial_{t} x}_{L^{2}(I_{n}; L^{2})}^{2}
        & = \norm{\sum_{i=0}^{k} x_{i}  \partial_{t} L_{n, i}^{\text{G}, 0}}_{L^{2}(I_{n}; L^{2})}^{2}
        \leq \left( \sum_{i=0}^{k} \norm{x_{i}}_{L^{2}(\Omega)}  \norm{\partial_{t} L_{n, i}^{\text{G}, 0}}_{L^{2}(I_{n})} \right)^{2} \\
        & \leq c \sum_{i=0}^{k} \norm{x_{i}}_{L^{2}(\Omega)}^{2}  \norm{\partial_{t} L_{n, i}^{\text{G}, 0}}_{L^{2}(I_{n})}^{2}
        \leq c \tau^{-2} \norm{x}_{L^{2}(I_{n}; L^{2})}^{2}.
    \end{align*}
\end{document}